\documentclass[oneside]{amsart}
\usepackage[T1]{fontenc}
\usepackage[latin9]{inputenc}
\usepackage{amsthm}
\usepackage{amstext}
\usepackage{amssymb}
\usepackage[all]{xy}

\makeatletter
\numberwithin{equation}{section}
\numberwithin{figure}{section}
\theoremstyle{plain}
\newtheorem{thm}{\protect\theoremname}[section]
  \theoremstyle{definition}
  \newtheorem{defn}[thm]{\protect\definitionname}
  \theoremstyle{plain}
  \newtheorem{lem}[thm]{\protect\lemmaname}
  \theoremstyle{remark}
  \newtheorem{rem}[thm]{\protect\remarkname}
  \theoremstyle{plain}
  \newtheorem{cor}[thm]{\protect\corollaryname}
  \theoremstyle{plain}
  \newtheorem{prop}[thm]{\protect\propositionname}
  \theoremstyle{definition}
  \newtheorem{example}[thm]{\protect\examplename}

\makeatother

  \providecommand{\corollaryname}{Corollary}
  \providecommand{\definitionname}{Definition}
  \providecommand{\examplename}{Example}
  \providecommand{\lemmaname}{Lemma}
  \providecommand{\propositionname}{Proposition}
  \providecommand{\remarkname}{Remark}
\providecommand{\theoremname}{Theorem}

\begin{document}

\title{Crossed Products of Banach Algebras. III.}
\begin{abstract}
In earlier work a crossed product of a Banach algebra was constructed
from a Banach algebra dynamical system $(A,G,\alpha)$ and a class
$\mathcal{R}$ of continuous covariant representations, and its representations
were determined. In this paper we adapt the theory to the ordered
context. We construct a pre-ordered crossed product of a Banach algebra
from a pre-ordered Banach algebra dynamical system $(A,G,\alpha)$
and a given uniformly bounded class $\mathcal{R}$ of continuous covariant
representations of $(A,G,\alpha)$. If $A$ has a positive bounded
approximate left identity and $\mathcal{R}$ consists of non-degenerate
continuous covariant representations, we establish a bijection between
the positive non-degenerate bounded representations of the pre-ordered
crossed product on pre-ordered Banach spaces with closed cones and
the positive non-degenerate $\mathcal{R}$-continuous covariant representations
of $(A,G,\alpha)$ on such spaces. Under mild conditions, we show
that this pre-ordered crossed product is the essentially unique pre-ordered
Banach algebra for which such a bijection exists. Finally, we study
pre-ordered generalized Beurling algebras. We show that they are bipositively
topologically isomorphic to pre-ordered crossed products of Banach
algebras associated with pre-ordered Banach algebra dynamical systems,
and hence the general theory allows us to describe their positive
representations on pre-ordered Banach spaces with closed cones.
\end{abstract}

\author{Marcel de Jeu}

\address{Marcel de Jeu, Mathematical Institute, Leiden University, P.O. Box
9512, 2300 RA Leiden, The Netherlands}

\email{mdejeu@math.leidenuniv.nl}

\author{Miek Messerschmidt}

\address{Miek Messerschmidt, Mathematical Institute, Leiden University, P.O.
Box 9512, 2300 RA Leiden, The Netherlands}

\email{mmesserschmidt@gmail.com}



\maketitle

\global\long\def\crossedprod{(A\rtimes_{\alpha}G)^{\mathcal{R}}}
\global\long\def\leftcent{\mathcal{M}_{l}}
\global\long\def\BeurlingTypeAlg{L^{1}(G,A,\omega;\alpha)}

\section{Introduction}

This paper is a continuation of \cite{2011arXiv1104.5151D} and \cite{CPII},
where, inspired by the theory of crossed products of $C^{*}$-algebras,
the theory of crossed products of Banach algebras is developed. The
lack of the convenient rigidity that $C^{*}$-algebras provide, where,
e.g., morphisms are automatically continuous and even contractive,
makes the task of developing the basics more laborious than it is
for crossed products of $C^{*}$-algebras. 

The paper \cite{2011arXiv1104.5151D} is for a large part concerned
with one result: the General Correspondence Theorem \cite[Theorem 8.1]{2011arXiv1104.5151D},
most of which is formulated as Theorem \ref{thm:General-Correspondence-Theorem}
below. With $(A,G,\alpha)$ a Banach algebra dynamical system and
$\mathcal{R}$ a uniformly bounded class of non-degenerate continuous
covariant representations of $(A,G,\alpha)$ on Banach spaces -- all
notions will be reviewed in Section \ref{sec:Preliminaries-and-recapitulation}
-- the General Correspondence Theorem, in the presence of a bounded
approximate left identity of $A$, yields a bijection between the
non-degenerate $\mathcal{R}$-continuous covariant representations
of $(A,G,\alpha)$, and the non-degenerate bounded representations
of the crossed product Banach algebra $\crossedprod$ associated with
$(A,G,\alpha)$ and $\mathcal{R}$.

In \cite{CPII} the theory established in \cite{2011arXiv1104.5151D}
is developed further. Amongst others, there it is shown that (under
mild conditions) the crossed product $\crossedprod$ is the unique
Banach algebra, up to topological isomorphism, which ``generates''
all non-degenerate $\mathcal{R}$-continuous covariant representations
of $(A,G,\alpha)$ \cite[Theorem 4.4]{CPII}. Furthermore, given a
weight $\omega$ on $G$ and assuming $\alpha$ is uniformly bounded,
for a particular choice of $\mathcal{R}$ it is shown that the crossed
product $\crossedprod$ is topologically isomorphic to a generalized
Beurling algebra $\BeurlingTypeAlg$ \cite[Section 5]{CPII}. These
algebras, as introduced in \cite{CPII}, are weighted Banach spaces
of (equivalence classes) of $A$-valued functions that are also associative
algebras with a multiplication that is continuous in both variables,
but they are not Banach algebras in general, since the norm need not
be submultiplicative. The General Correspondence Theorem then provides
a bijection between the non-degenerate continuous covariant representations
of $(A,G,\alpha)$, of which the representation of $G$ is bounded
by a multiple of $\omega$, and the non-degenerate bounded representations
of $\BeurlingTypeAlg$ \cite[Theorem 5.20]{CPII}. When $A$ is taken
to be the scalars, generalized Beurling algebras reduce to classical
Beurling algebras, which are true Banach algebras, and then \cite[Corollary 5.22]{CPII}
describes their non-degenerate bounded representations. In the case
where $\omega=1$ as well, this specializes to the classical bijection
between uniformly bounded representations of $G$ on Banach spaces
and non-degenerate bounded representations of $L^{1}(G)$ (cf.\,\cite[Assertion VI.1.32]{Helemski}).

In the current paper we adapt the theory developed in \cite{2011arXiv1104.5151D}
and \cite{CPII} to the ordered context: that of pre-ordered Banach
spaces and algebras. Apart from its intrinsic interest, this is also
motivated by the proven relevance of crossed products of $C^{*}$-algebras
for unitary group representations. As is well known, a decomposition
of a general unitary group representation into a direct integral of
irreducible unitary representations is obtained via the group $C^{*}$-algebra
(a particularly simple crossed product), and Mackey's Imprimitivity
Theorem can, by Rieffel's work, now be conceptually interpreted in
terms of (strong) Morita equivalence of a crossed product of a $C^{*}$-algebra
and a group $C^{*}$-algebra. We hope that the results in the present
paper will contribute to similar developments in the theory of positive
representations of groups on pre-ordered Banach spaces (and Banach
lattices in particular), which exist in abundance.

We are mainly concerned with four topics: Firstly, an adaptation of
the construction of crossed products of Banach algebras from \cite{2011arXiv1104.5151D}
to the ordered context (cf.\,Section \ref{sec:Pre-ordered-Banach-algebra-dynamical-systems-and-crossed-prods}).
Secondly, proving a version of the General Correspondence Theorem
in this context (cf.\,Theorem \ref{thm:Ordered-general-correspondence}).
Thirdly, for a pre-ordered Banach algebra dynamical system $(A,G,\alpha)$
and uniformly bounded class of positive continuous covariant representations
$\mathcal{R}$, we establish (under mild conditions) the uniqueness,
up to bipositive topological isomorphism, of the associated pre-ordered
crossed product $\crossedprod$ as the unique pre-ordered Banach algebra
which ``generates'' all positive non-degenerate $\mathcal{R}$-continuous
covariant representations of $(A,G,\alpha)$ (cf.\,Theorem \ref{thm:universal-property}).
And fourthly, we describe the positive non-degenerate bounded representations
of a pre-ordered Beurling algebra $\BeurlingTypeAlg$ in terms of
positive non-degenerate continuous covariant representations of the
pre-ordered Banach algebra dynamical system $(A,G,\alpha)$ to which
$\BeurlingTypeAlg$ is associated (cf.\,Section \ref{sec:Pre-ordered-generalized-Beurling}).

We now briefly describe the structure of the paper.

Section \ref{sec:Preliminaries-and-recapitulation} contains all preliminary
definitions and results concerning pre-ordered vector spaces and crossed
products. Specifically, Sections \ref{sub:Ordered-vector-spaces}--\ref{sub:positive-Representations-of-groups-and-algebras}
provide preliminary definitions and results concerning pre-ordered
vector spaces and algebras and pre-ordered normed spaces and algebras.
Some of the material is completely standard and/or elementary, but
since the fields of representation theory and positivity seem to be
somewhat disjoint, we have included it in an attempt to enhance the
accessibility of this paper, which draws on both disciplines. Sections
\ref{sub:unordered-basic-definitions} and \ref{sub:classical-Correspondence}
provide a brief recapitulation of all relevant notions from \cite{2011arXiv1104.5151D}
relating to Banach algebra dynamical systems and crossed products.

In Section \ref{sec:Pre-ordered-Banach-algebra-dynamical-systems-and-crossed-prods}
we define pre-ordered Banach algebra dynamical systems and provide
the construction of pre-ordered crossed products associated with such
systems. The construction is largely the same as in the general unordered
case, but differs in keeping track of how order structures of $(A,G,\alpha)$
and $\mathcal{R}$ induce a natural cone, denoted $\crossedprod_{+}$,
which defines a pre-order on the crossed product $\crossedprod$.
Theorem \ref{thm:ordered-crossed-prod-properties} collects properties
of the cone $\crossedprod_{+}$ (and thereby the order structure)
of a pre-ordered crossed product $\crossedprod$ in terms of the order
properties of the pre-ordered Banach algebra dynamical system $(A,G,\alpha)$
and uniformly bounded class $\mathcal{R}$ of continuous covariant
representations. Finally, we adapt the General Correspondence Theorem
to the ordered context. In the presence of a positive bounded approximate
left identity of $A$, it gives a canonical bijection between the
positive non-degenerate $\mathcal{R}$-continuous covariant representations
of $(A,G,\alpha)$ on pre-ordered Banach spaces with closed cones,
and positive non-degenerate bounded representations of the pre-ordered
crossed product $\crossedprod$ on such spaces (cf.\,Theorem \ref{thm:Ordered-general-correspondence}).

Paralleling work of Raeburn's \cite{RaeburnOriginalUniversalPaper},
in Section \ref{sec:Uniqueness-of-the-crossed-prod} we show that
(under mild additional hypotheses) the pre-ordered crossed product
$\crossedprod$ associated with $(A,G,\alpha)$ and $\mathcal{R}$
is the unique pre-ordered Banach algebra, up to bipositive topological
isomorphism, which ``generates'' all positive non-degenerate $\mathcal{R}$-continuous
covariant representations of $(A,G,\alpha)$ (cf.\,Theorem \ref{thm:universal-property}).

Finally, in Section \ref{sec:Pre-ordered-generalized-Beurling}, we
will study pre-ordered generalized Beurling algebras $\BeurlingTypeAlg$.
These algebras can be defined for any pre-ordered Banach algebra dynamical
system $(A,G,\alpha)$ and weight $\omega$ on $G$, provided that
$\alpha$ is uniformly bounded. If $A$ has a bounded approximate
right identity, for a specific choice of $\mathcal{R}$ the pre-ordered
crossed product $\crossedprod$ is shown to be bipositively topologically
isomorphic to a pre-ordered generalized Beurling algebra $\BeurlingTypeAlg$
(cf.\,Theorem \ref{thm:Choosing-R-correctly-Crossed-Products-are-beurling}).
In the presence of a positive bounded approximate left identity of
$A$, our ordered version of the General Correspondence Theorem, Theorem
\ref{thm:Ordered-general-correspondence}, then provides a bijection
between the positive non-degenerate bounded representations of $\BeurlingTypeAlg$
and the positive non-degenerate continuous covariant representations
of the pre-ordered Banach algebra dynamical system $(A,G,\alpha)$,
where the representation of the group $G$ is bounded by a multiple
of $\omega$ (cf.\,Theorem \ref{thm:continuous-non-deg-covars-are-R-continuous}).
In the case where $A$ is a Banach lattice algebra, it is shown that
$\BeurlingTypeAlg$ also becomes a Banach lattice (although it is
not generally a Banach algebra), and, under further conditions, becomes
a Banach lattice algebra (cf.\,Theorem \ref{thm:order-structures-of-beurling}).
In the simplest case, where $A$ is taken to be the real numbers and
$\omega=1$, our results reduce to a bijection between the positive
strongly continuous uniformly bounded representations of $G$ on pre-ordered
Banach spaces with closed cones on the one hand, and the positive
non-degenerate bounded representations of $L^{1}(G)$ on such spaces
on the other hand; this also follows from \cite[Assertion VI.1.32]{Helemski}.

\section{Preliminaries and recapitulation\label{sec:Preliminaries-and-recapitulation}}

In this section we will introduce the terminology and notation used
in the rest of the paper and give a brief recapitulation of Banach
algebra dynamical systems and their crossed products. Sections \ref{sub:Ordered-vector-spaces}--\ref{sub:positive-Representations-of-groups-and-algebras}
will introduce general notions concerning pre-ordered (normed) vector
spaces and algebras. Sections \ref{sub:unordered-basic-definitions}
and \ref{sub:classical-Correspondence} will give a brief overview
of results from \cite{2011arXiv1104.5151D} on Banach algebra dynamical
systems and their crossed products.

Throughout this paper all vector spaces are assumed to be over the
reals, and all locally compact topologies are assumed to be Hausdorff. 

Let $X$ and $Y$ be normed spaces. The normed space of bounded linear
operators from $X$ to $Y$ will be denoted by $B(X,Y)$, and by $B(X)$
if $X=Y$. The group of invertible elements in $B(X)$ will be denoted
by $\mbox{Inv}(X)$. If $A$ is a normed algebra, by $\mbox{Aut}(A)$
we will denote its group of bounded automorphisms. We do not assume
algebras to be unital.

For a locally compact topological space $\Omega$ and topological
vector space $V$, we will denote the space of all continuous compactly
supported functions on $\Omega$ taking values in $V$ by $C_{c}(\Omega,V)$.
If $V=\mathbb{R}$, we write $C_{c}(\Omega)$ for $C_{c}(\Omega,\mathbb{R})$.

If $G$ is a locally compact group, we will denote its identity element
by $e\in G$. For $f\in C_{c}(G)$, we will write $\int_{G}f(s)\, ds$
for the integral of $f$ with respect to a fixed left Haar measure
$\mu$ on $G$.

\subsection{Pre-ordered vector spaces and algebras\label{sub:Ordered-vector-spaces}}

We introduce the following terminology.

Let $V$ be a vector space. A subset $C\subseteq V$ will be called
a \emph{cone }if $C+C\subseteq C$ and $\lambda C\subseteq C$ for
all $\lambda\geq0$. The pair $(V,C)$ will be called a \emph{pre-ordered
vector space} and, for $x,y\in V$, by $y\leq x$ we mean $x-y\in C$.
Elements of $C$ will be called \emph{positive.} We will often suppress
mention of the cone $C$, and merely say that $V$ is a pre-ordered
vector space. In this case, we will denote the implicit cone by $V_{+}$
and refer to it as \emph{the cone of $V$}. A cone $C\subseteq V$
will be said to be a \emph{proper cone }if $C\cap(-C)=\{0\}$, in
which case $\leq$ is a partial order, and then \emph{$(V,C)$ }will
be called\emph{ }an\emph{ ordered vector space.} A cone $C\subseteq V$
will be said to be \emph{generating} (\emph{in $V$}) if $V=C-C$.
If $(V,C)$ is a pre-ordered vector space and $V$ is also an associative
algebra such that $C\cdot C\subseteq C$, we will say $(V,C)$ is
a \emph{pre-ordered algebra.}

If $(V_{1},C_{1})$ and $(V_{2},C_{2})$ are pre-ordered vector spaces,
we will say a linear map $T:V_{1}\to V_{2}$ is \emph{positive} if
$TC_{1}\subseteq C_{2}$. If $T$ is injective and both $TC_{1}\subseteq C_{2}$
and $T^{-1}C_{2}\subseteq C_{1}$ hold, we will say $T$ is \emph{bipositive.}

With $W\subseteq V$ a subspace and $q:V\to V/W$ the quotient map,
$q(C)\subseteq V/W$ will be called the \emph{quotient cone}. Then
$(V/W,q(C))$ is a pre-ordered vector space. Clearly $q:V\to V/W$
is positive and $q(C)$ is generating in $V/W$ if $C$ is generating
in $V$.

\subsection{Pre-ordered normed spaces and algebras\label{sub:Pre-ordered-normed-spaces-and-algebras}}

We give a brief description of pre-ordered normed vector spaces and
algebras. In Section \ref{sec:Pre-ordered-Banach-algebra-dynamical-systems-and-crossed-prods}
we will apply the results from this section to describe the order
structure of crossed products associated with pre-ordered Banach algebra
dynamical systems.

If $A$ is a pre-ordered algebra that is also a normed algebra, then
we will call $A$ a \emph{pre-ordered normed algebra,} and a \emph{pre-ordered
Banach algebra} if $A$ is complete. The \emph{positive automorphism
group }(\emph{of $A$}) is defined by $\mbox{Aut}_{+}(A):=\{\alpha\in\mbox{Aut}(A):\alpha^{\pm1}(A_{+})\subseteq A_{+}\}\subseteq B(A)$.

Let $X$ and $Y$ be pre-ordered normed spaces. We will always assume
that $B(X,Y)$ is endowed with the \emph{natural operator cone} $B(X,Y)_{+}:=\{T\in B(X,Y):TX_{+}\subseteq Y_{+}\}$,
so that $B(X,Y)$ is a pre-ordered normed space, and $B(X)$ is a
pre-ordered normed algebra. We define the \emph{group of bipositive
invertible operators }on $X$ by $\mbox{Inv}_{+}(X):=\{T\in\mbox{Inv}(X):T^{\pm1}X_{+}\subseteq X_{+}\}$.
We will say $X_{+}$ is \emph{topologically generating} (\emph{in
$X$}) if $X=\overline{X_{+}-X_{+}}$. If the ordering defined by
$X_{+}$ is a lattice-ordering (i.e., if every pair of elements from
$X$ has a supremum, denoted by $\vee$) we will call $X$ a \emph{normed
vector lattice} if $|x|\leq|y|$ implies $\|x\|\leq\|y\|$ for all
$x,y\in X$, where $|x|:=x\vee(-x)$. A complete normed vector lattice
will be called a \emph{Banach lattice}. A pre-ordered Banach algebra
that is also a Banach lattice will be called a \emph{Banach lattice
algebra. }A subspace $Y\subseteq X$ in a vector lattice $X$ is called
an \emph{order ideal }if, for $g\in Y$ and $f\in X$, $|f|\leq|g|$
implies $f\in Y$.

We will need completions of pre-ordered normed spaces in Section \ref{sec:Pre-ordered-Banach-algebra-dynamical-systems-and-crossed-prods},
to be able to describe pre-ordered crossed products associated with
pre-ordered Banach algebra dynamical systems.
\begin{defn}
\label{def:completion-or-ordered-normed-space}Let $V$ be a pre-ordered
normed space. We define \emph{the completion of $V$ }by $(\overline{V},\overline{V_{+}})$,
where $\overline{V}$ denotes the usual metric completion of the normed
space $V$, and $\overline{V_{+}}$ the closure of $V_{+}$ in $\overline{V}$.
\end{defn}
The following two elementary observations are included for later reference.
\begin{lem}
\label{lem:bounded-positive-operator-on-ordered-noremed-space-implies-extension-is-positive}Let
$V$ be a pre-ordered normed space, $X$ a pre-ordered Banach space
with a closed cone, and $T:V\to X$ a positive bounded linear operator.
Then the bounded extension of $T$ to the completion of $V$ is a
positive operator.
\end{lem}
{}
\begin{lem}
\label{lem:completion-of-normed-algebra-isbanach-algebra}If $V$
is a pre-ordered normed algebra, then its completion is a pre-ordered
Banach algebra with a closed cone.
\end{lem}
Together with Corollary \ref{thm:C(Omega,X+)-generating-if-X+-is-generating},
the following two elementary results will be used in Theorem \ref{thm:ordered-crossed-prod-properties}
to give sufficient conditions for the cone of a crossed product of
a pre-ordered Banach algebra dynamical system to be (topologically)
generating.
\begin{lem}
\label{lem:generating-becomes-weakly-generating-in-completions}Let
$V$ be a pre-ordered normed space. If $V_{+}$ is topologically generating
in $V$, then $V_{+}$, and hence the cone $\overline{V_{+}}$, is
topologically generating in the completion $\overline{V}$.\end{lem}
\begin{proof}
Let $w\in\overline{V}$ be arbitrary, and let $(v_{n})\subseteq V$
be such that $v_{n}\to w$. For every $n\in\mathbb{N}$, let $a_{n},b_{n}\in V_{+}$
be such that $\|v_{n}-(a_{n}-b_{n})\|<2^{-n}$. Since
\[
\|w-(a_{n}-b_{n})\|\leq\|w-v_{n}\|+\|v_{n}-(a_{n}-b_{n})\|<\|w-v_{n}\|+2^{-n},
\]
$(a_{n}-b_{n})\subseteq V_{+}-V_{+}$ converges to $w$.
\end{proof}
In certain cases the conclusion of the previous lemma for $\overline{V_{+}}$
may be strengthened. 
\begin{lem}
\label{lem:generating-with-continuous-pos-part-function-becomes-generating-in-completions}Let
$V$ be a pre-ordered normed space and $(\cdot)^{+}:V\to V_{+}$ a
function such that $v\leq v^{+}$ for all $v\in V$. Then $V_{+}$
is generating in $V$, and if $(\cdot)^{+}:V\to V_{+}$ maps Cauchy
sequences to Cauchy sequences, then the cone $\overline{V_{+}}$ is
generating in the completion $\overline{V}$.\end{lem}
\begin{proof}
It is obvious that the fact that $V_{+}$ is generating in $V$ is
equivalent with the existence of a function $(\cdot)^{+}:V\to V_{+}$
such that $v\leq v^{+}$ for all $v\in V$.

Assuming that $(\cdot)^{+}:V\to V_{+}$ maps Cauchy sequences to Cauchy
sequences, let $w\in\overline{V}$ be arbitrary and let $(v_{n})\subseteq V$
be such that $v_{n}\to w$. The sequence $(v_{n})\subseteq V$ is
Cauchy, hence, by hypothesis, so is $(v_{n}^{+})\subseteq V_{+}\subseteq\overline{V_{+}}$.
Since $\overline{V_{+}}$ is closed in $\overline{V}$, $(v_{n}^{+})$
converges to some $w'\in\overline{V_{+}}$. Since $v_{n}^{+}-v_{n}\in V_{+}\subseteq\overline{V_{+}}$,
we have $w'-w=\lim_{n\to\infty}(v_{n}^{+}-v_{n})\in\overline{V_{+}}$.
Writing $w=w'-(w'-w)$ yields the result.\end{proof}
\begin{rem}
If $V$ is a normed vector lattice, then the map $v\mapsto v\vee0$
is uniformly continuous and hence maps Cauchy sequences to Cauchy
sequences \cite[Proposition II.5.2]{Schaefer}. Hence $\overline{V_{+}}$
is generating in the completion $\overline{V}$. Since, in this case,
$\overline{V}$ is actually a Banach lattice \cite[Corollary 2, p.\,84]{Schaefer},
this is not unexpected.
\end{rem}
The following refinement of And\^o's Theorem \cite[Lemma 1]{Ando}
is a special case of \cite[Theorem 4.1]{RightInverses}, of which
the essence is that the decomposition of elements as the difference
of positive elements can be chosen in a bounded, continuous and positively
homogeneous manner. Its proof proceeds through applications of a generalization
of the usual Open Mapping Theorem \cite[Theorem 3.2]{RightInverses}
and the Michael Selection Theorem \cite[Theorem 17.66]{AliprantisBorder}.
It will be applied in Theorem \ref{thm:ordered-crossed-prod-properties}
to prove that the cones of certain crossed products associated with
pre-ordered Banach algebras are topologically generating.
\begin{thm}
\label{thm:upper-bound-function}Let $X$ be a pre-ordered Banach
space with closed generating cone. Then there exist a constant $\alpha>0$
and continuous positively homogeneous maps $(\cdot)^{\pm}:X\to X_{+}$
such that $x=x^{+}-x^{-}$ and $\|x^{+}\|+\|x^{-}\|\leq\alpha\|x\|$
for all $x\in X$.
\end{thm}
{}Simply through composition with the functions $(\cdot)^{\pm}:X\to X_{+}$,
cones of continuous $X_{+}$-valued functions are then immediately
seen to be generating in spaces of continuous $X$-valued functions.
For example:
\begin{cor}
\label{thm:C(Omega,X+)-generating-if-X+-is-generating}Let $\Omega$
be a locally compact Hausdorff space and $X$ be a pre-ordered Banach
space with closed generating cone. Then the cone $C_{c}(\Omega,X_{+})$
is generating in $C_{c}(\Omega,X)$. In fact, there exists a constant
$\alpha>0$ with the property that, for every $f\in C_{c}(G,X)$,
there exist $f^{\pm}\in C_{c}(G,X)$ such that $\|f^{+}(\omega)\|+\|f^{-}(\omega)\|\leq\alpha\|f(\omega)\|$
for all $\omega\in\Omega$. In particular, $\|f^{\pm}\|_{\infty}\leq\alpha\|f\|_{\infty}$
and $\textup{supp}(f^{\pm})\subseteq\textup{supp}(f).$\end{cor}
\begin{rem}
\label{rem:wickstead}The earliest results of this type known to the
authors are \cite[Theorem 2.3]{AsimowAtkinson} and \cite[Theorem 4.4]{Wickstead}.
Both results proceed through an application of Lazar's affine selection
theorem to show that canonical cones of certain spaces of continuous
affine functions are generating. The result \cite[Theorem 2.3]{AsimowAtkinson}
shows, with $K$ a Choquet simplex and $X$ a pre-ordered Banach space
with a closed cone, that the space $A(K,X)$ of continuous affine
functions from $K$ to $X$ has $A(K,X_{+})$ as a generating cone.
By taking $K$ to be the regular Borel probability measures on a compact
Hausdorff space $\Omega$, this result includes the case that $C(\Omega,X_{+})$
is generating in $C(\Omega,X)$, which is part of the statement of
Corollary \ref{thm:C(Omega,X+)-generating-if-X+-is-generating}. 
\end{rem}
We will now define normality and conormality properties for a pre-ordered
Banach space $X$ with a closed cone, and subsequently show in Theorem
\ref{thm-yamamuro-results} how these properties imply normality properties
of the pre-ordered normed space $B(X,Y)$. In Theorem \ref{thm:ordered-crossed-prod-properties}
this will be used to conclude (conditional) normality properties of
a pre-ordered crossed product.
\begin{defn}
Let $X$ be a pre-ordered Banach space with closed cone and $\alpha>0$.
We define the following \textbf{\emph{normality properties}}\emph{:}
\begin{enumerate}
\item We will say that $X$ is \emph{$\alpha$-normal} if, for any $x,y\in X$,
$0\leq x\leq y$ implies $\|x\|\leq\alpha\|y\|$. 
\item We will that $X$ is \emph{$\alpha$-absolutely normal} if, for any
$x,y\in X$, $\pm x\leq y$ implies $\|x\|\leq\alpha\|y\|$. 
\end{enumerate}
We define the following \textbf{\emph{conormality properties}}\emph{:}
\begin{enumerate}
\item We will say that $X$ is \emph{approximately $\alpha$-absolutely
conormal} if, for any $x\in X$ and $\varepsilon>0$, there exists
some $a\in X_{+}$ such that $\pm x\leq a$ and $\|a\|<\alpha\|x\|+\varepsilon$. 
\item We will say that $X$ is \emph{approximately $\alpha$-sum-conormal}
if, for any $x\in X$ and $\varepsilon>0$, there exist some $a,b\in X_{+}$
such that $x=a-b$ and $\|a\|+\|b\|<\alpha\|x\|+\varepsilon$. 
\end{enumerate}
\end{defn}
\begin{rem}
Normality (terminology due to Krein \cite{KreinNormality}) and (approximate)
conormality (terminology due to Walsh \cite{Walsh}) are dual properties
for pre-ordered Banach spaces with closed cones. Roughly speaking,
a pre-ordered Banach space with a closed cone has some normality property
precisely if its dual has a corresponding conormality property, and
vice versa. The most complete reference for such duality relationships
seems to be \cite{BattyRobinson}. 

For a pre-ordered Banach space $X$ with a closed cone, elementary
arguments will show that $\alpha$-absolutely normality of $X$ implies
that $X$ is $\alpha$-normal, which in turn implies that $X_{+}$
is a proper cone. Also, approximate $\alpha$-sum-conormality of $X$
implies that $X$ is approximately $\alpha$-absolute conormal, which
in turn implies that $X_{+}$ is generating in $X$. An application
of And\^o's Theorem \cite[Lemma 1]{Ando} shows conversely that,
if $X_{+}$ is generating in $X$, then there exists some $\beta>0$
such that, for every $x\in X$, there exists $a,b\in X_{+}$ such
that $x=a-b$ and $\max\{\|a\|,\|b\|\}\leq\beta\|x\|$ (another form
of conormality, clearly implying approximate $2\beta$-sum-conormality).
We note that Banach lattices are always $1$-absolutely normal and
approximately $1$-absolutely conormal.
\end{rem}
The following results relate normality properties of spaces of operators
to the normality and conormality properties of the underlying spaces.
Part (2) is due to Wickstead \cite[Theorem 3.1]{Wickstead}. Part
(3) is a slight refinement of a result due to Yamamuro \cite[Theorem 1.3]{Yamamuro},
where it is proven for the case $X=Y$ and $\alpha=\beta=1$. No reference
for part (4) is known to the authors. We include proofs for convenience
of the reader.
\begin{thm}
\label{thm-yamamuro-results}Let $X$ and $Y$ be pre-ordered Banach
spaces with closed cones and $\alpha,\beta>0$. 
\begin{enumerate}
\item If $X_{+}$ is generating and $Y_{+}$ is a proper cone, then $B(X,Y)_{+}$
is a proper cone.
\item If $X_{+}$ is generating and $Y$ is $\alpha$-normal, then there
exists some $\gamma>0$ for which $B(X,Y)$ is $\gamma$-normal.
\item If $X$ is approximately $\alpha$-absolutely conormal and $Y$ is
$\beta$-absolutely normal, then $B(X,Y)$ is $\alpha\beta$-absolutely
normal.
\item If $X$ is approximately $\alpha$-sum-conormal and $Y$ is $\beta$-normal,
then $B(X,Y)$ is $\alpha\beta$-normal.
\end{enumerate}
\end{thm}
\begin{proof}
We prove (1). Let $T\in B(X,Y)_{+}\cap(-B(X,Y)_{+})$. If $x\in X_{+}$,
then $Tx\geq0$ and $(-T)x\geq0$. Hence $Tx=0$, since $Y_{+}$ is
proper. Since $X_{+}$ is generating, we have $T=0$ as required.

We prove (2). By And\^o's Theorem \cite[Lemma 1]{Ando}, the fact
that $X_{+}$ is generating in $X$ implies that there exists some
$\beta>0$ such that, for every $x\in X$, there exist $a,b\in X_{+}$
such that $x=a-b$ and $\max\{\|a\|,\|b\|\}\leq\beta\|x\|$. Let $T,S\in B(X,Y)$
be such that $0\leq T\leq S$. Then, for any $x\in X$, let $a,b\in X_{+}$
be such that $x=a-b$ and $\max\{\|a\|,\|b\|\}\leq\beta\|x\|$, so
that $0\leq Ta\leq Sa$ and $0\leq Tb\leq Sb$. By $\alpha$-normality
of $Y$, 
\[
\|Tx\|\leq\|Ta\|+\|Tb\|\leq\alpha(\|Sa\|+\|Sb\|)\leq\alpha\|S\|(\|a\|+\|b\|)\leq2\alpha\beta\|S\|\|x\|,
\]
hence $\|T\|\leq2\alpha\beta\|S\|$.

We prove (3). Let $T,S\in B(X,Y)$ satisfy $\pm T\leq S$. Let $x\in X$
be arbitrary. Then, for every $\varepsilon>0$, there exists some
$a\in X_{+}$ satisfying $\pm x\leq a$ and $\|a\|<\alpha\|x\|+\varepsilon$.
Then 
\[
Tx=T\left(\frac{a+x}{2}\right)-T\left(\frac{a-x}{2}\right),
\]
and hence 
\[
\pm Tx=\pm T\left(\frac{a+x}{2}\right)\mp T\left(\frac{a-x}{2}\right).
\]
Since $(a+x)/2\geq0$, $(a-x)/2\geq0$ and $\pm T\leq S$, we find
\[
\pm Tx\leq S\left(\frac{a+x}{2}\right)+S\left(\frac{a-x}{2}\right)=Sa.
\]
Now, because $Y$ is $\beta$-absolutely normal, we obtain 
\[
\|Tx\|\leq\beta\|Sa\|\leq\beta\|S\|\|a\|\leq\alpha\beta\|S\|\|x\|+\varepsilon\beta\|S\|.
\]
Since $\varepsilon>0$ was chosen arbitrarily, we conclude that $B(X,Y)$
is $\alpha\beta$-absolutely normal.

We prove (4). Let $T,S\in B(X,Y)$ satisfy $0\leq T\leq S$. Let $x\in X$
be arbitrary. Then, for every $\varepsilon>0$, there exist $a,b\in X_{+}$
such that $x=a-b$ and $\|a\|+\|b\|\leq\alpha\|x\|+\varepsilon$.
Because $0\leq T\leq S$, we have $0\leq Ta\leq Sa$ and $0\leq Tb\leq Sb$.
Since $Y$ is $\beta$-normal, we obtain
\[
\|Tx\|\leq\|Ta\|+\|Tb\|\leq\beta\|Sa\|+\beta\|Sb\|\leq\beta\|S\|\left(\|a\|+\|b\|\right)\leq\alpha\beta\|S\|\|x\|+\varepsilon\beta\|S\|.
\]
 Since $\varepsilon>0$ was chosen arbitrarily, we conclude that $B(X,Y)$
is $\alpha\beta$-normal.\end{proof}
\begin{rem}
\label{remark:onyamamuroresults}We note some specific cases of the
above theorem. Any Banach lattice $X$ is both approximately $1$-absolutely
conormal and $1$-absolutely normal, therefore (3) in the previous
result implies that $B(X)$ is always $1$-absolutely normal in this
case. Also, if $X$ is a Banach lattice and $Y$ and $Z$ are pre-ordered
Banach spaces with closed cones that are respectively $\alpha$-absolutely
normal and approximately $\alpha$-absolutely conormal for some $\alpha>0$,
then $B(X,Y)$ and $B(Z,X)$ are $\alpha$-absolutely normal, again
by (3) above. If a Banach lattice $X$ is an AL-space (i.e., $\|x+y\|=\|x\|+\|y\|$
for all $x,y\in X_{+}$), then it is approximately $1$-sum-conormal,
and hence, for any $\alpha$-normal pre-ordered Banach space $Y$
with a closed cone, $B(X,Y)$ is $\alpha$-normal by (4) in the previous
result. 

Let $\mathcal{H}$ be real Hilbert space endowed with a Lorentz cone
$\mathcal{L}_{v}:=\{x\in\mathcal{H}:\langle v|x\rangle\geq\|Px\|\}$,
where $v\in\mathcal{H}$ is any norm 1 element, and $P$ the projection
onto $\{v\}^{\bot}$. Although not a Banach lattice if $\dim\mathcal{H}\geq3$,
the ordered Banach space $(\mathcal{H},\mathcal{L}_{v})$ is $1$-absolutely
normal and approximately $1$-absolutely conormal \cite{PosAttainedOperatorNorms}.
Hence, again by (3) above, $B(\mathcal{H})$ is $1$-absolutely normal,
and if $Y$ and $Z$ are pre-ordered Banach spaces with closed cones
that are respectively $\alpha$-absolutely normal and approximately
$\alpha$-absolutely conormal for some $\alpha>0$, then $B(\mathcal{H},Y)$
and $B(Z,\mathcal{H})$ are $\alpha$-absolutely normal.
\end{rem}

\subsection{Representations on pre-ordered normed spaces\label{sub:positive-Representations-of-groups-and-algebras}}

We will now introduce positive representations of groups and pre-ordered
normed algebras on pre-ordered normed spaces. In Section \ref{sec:Pre-ordered-Banach-algebra-dynamical-systems-and-crossed-prods}
we will use the notions in this section to describe a bijection between
the positive non-degenerate bounded representations of a crossed product
associated with a pre-ordered Banach algebra dynamical system on the
one hand, and positive non-degenerate covariant representations of
the certain dynamical system on the other hand (cf.\,Theorem \ref{thm:Ordered-general-correspondence}).
\begin{defn}
Let $A$ be a normed algebra and $X$ a normed space. An algebra homomorphism
$\pi:A\to B(X)$ will be called a \emph{representation of $A$ on}
$X$. We will write $X_{\pi}$ for $X$, if the connection between
$X_{\pi}$ and $\pi$ requires emphasis. We will say that $\pi$ is
\emph{non-degenerate} if $\mbox{span}\{\pi(a)x:a\in A,x\in X\}$ is
dense in $X$. 

If $A$ is a pre-ordered normed algebra and $X$ is a pre-ordered
normed space, we will say that a representation $\pi$ of $A$ on
$X$ is \emph{positive} if $\pi(A_{+})\subseteq B(X)_{+}$.
\end{defn}
{}
\begin{defn}
Let $G$ be a locally compact group and $X$ a normed space. A group
homomorphism $U:G\to\mbox{Inv}(X)$ will be called a \emph{representation}
(\emph{of $G$ on $X$}). 

If $X$ is a pre-ordered normed space, a group homomorphism $U:G\to\mbox{Inv}_{+}(X)\subseteq B(X)$
(cf.\,Section \ref{sub:Pre-ordered-normed-spaces-and-algebras})
will be called a \emph{positive representation} of $G$ on $X$. 
\end{defn}
For typographical reasons we will write $U_{s}$ instead of $U(s)$
for $s\in G$. 

Note that continuity is not included in the definition of representations
of normed algebras and locally compact groups, and that representations
of a unital algebra are not required to be unital.

The left centralizer algebra of a normed algebra, to be introduced
next, plays a crucial role in the construction of the bijection mentioned
in the first paragraph of this section. 
\begin{defn}
Let $A$ be a normed algebra. A bounded linear operator $L:A\to A$
will be called a \emph{left centralizer of }$A$ if $L(ab)=(La)b$
for all $a,b\in A$. The unital normed algebra of all left centralizers,
with the operator norm inherited from $B(A)$, will be denoted $\leftcent(A)$
and called the \emph{left centralizer algebra of $A$}. The \emph{left
regular representation of $A$,} $\lambda:A\to\leftcent(A)$, is defined
by $\lambda(a)b:=ab$ for $a,b\in A$. 

If $A$ is a pre-ordered normed algebra, we will always assume that
$\leftcent(A)$ is endowed with the cone $\leftcent(A)\cap B(A)_{+}$.
Then $\lambda:A\to B(A)$ is a positive contractive representation
of $A$ on itself.
\end{defn}
{}
\begin{defn}
If $A$ is a pre-ordered normed algebra, we will say an approximate
left (right) identity $(u_{i})$ of $A$ is \emph{positive }if $(u_{i})\subseteq A_{+}$.
\end{defn}
The result \cite[Theorem 4.1]{2009arXiv0904.3268D} plays a key role
in the proof of the General Correspondence Theorem (Theorem \ref{thm:General-Correspondence-Theorem}).
We collect the relevant parts in Theorem \ref{thm:extension-theorem},
including how it can be applied to representations of pre-ordered
normed algebras on pre-ordered Banach spaces with closed cones. This
will be used to adapt the General Correspondence Theorem to the ordered
context (cf.\,Theorem \ref{thm:Ordered-general-correspondence}).
\begin{thm}
\label{thm:extension-theorem} Let $B$ be a normed algebra with an
$M$-bounded approximate left identity $(u_{i})$ and $X$ a Banach
space. If $T:B\to B(X)$ is a non-degenerate bounded representation,
then the map $\overline{T}:\leftcent(B)\to B(X)$ defined by $\overline{T}(L):=\textup{SOT-lim}_{i}T(Lu_{i})$
is the unique representation of $\leftcent(B)$ on $X$ such that
the diagram 
\[
\xymatrix{B\ar[r]^{T}\ar[dr]^{\lambda} & B(X)\\
 & \leftcent(B)\ar[u]_{\overline{T}}
}
\]
commutes. Moreover, $\overline{T}$ is non-degenerate and bounded,
with $\|\overline{T}\|\leq M\|T\|$, and satisfies $\overline{T}(L)T(b)=T(Lb)$
for all $b\in B$ and $L\in\leftcent(B)$. 

If, in addition, $B$ is a pre-ordered normed algebra, $(u_{i})$
is positive, and $X$ is a pre-ordered Banach space with a closed
cone, then $\overline{T}:\leftcent(B)\to B(X)$ is a positive non-degenerate
bounded representation of $\leftcent(B)$ on $X$.
\end{thm}

\subsection{Banach algebra dynamical systems and crossed products\label{sub:unordered-basic-definitions}}

We recall some basic definitions and results from \cite{2011arXiv1104.5151D}.
\begin{defn}
Let $A$ be a Banach algebra, $G$ a locally compact group, and $\alpha:G\to\mbox{Aut}(A)$
a strongly continuous representation of $G$ on $A$. Then we will
call the triple $(A,G,\alpha)$ a \emph{Banach algebra dynamical system. }
\end{defn}
If $(A,G,\alpha)$ is a Banach algebra dynamical system, $C_{c}(G,A)$
can be made into an associative algebra by defining the twisted convolution
\[
(f*g)(s):=\int_{G}f(r)\alpha_{r}(g(r^{-1}s))\, dr\quad(f,g\in C_{c}(G,A),\ s\in G).
\]
Here, as in \cite{2011arXiv1104.5151D}, integrals of compactly supported
continuous Banach space valued functions are defined by duality, following
\cite[Section 3]{Rudin}. 

Let $(A,G,\alpha)$ be a Banach algebra dynamical system and $X$
a normed space. If $\pi:A\to B(X)$ and $U:G\to\mbox{Inv}(X)$ are
representations satisfying 
\[
\pi(\alpha_{s}(a))=U_{s}\pi(a)U_{s}^{-1}\quad(a\in A,\ s\in G),
\]
 we will say that $(\pi,U)$ is a \emph{covariant representation }of
$(A,G,\alpha)$ on $X$. We will say $(\pi,U)$ is \emph{continuous}
if $\pi$ is bounded and $U$ is strongly continuous. We will say
that $(\pi,U)$ is \emph{non-degenerate }if $\pi$ is non-degenerate.

If $(\pi,U)$ is a continuous covariant representation of $(A,G,\alpha)$
on a Banach space $X$, then, as in \cite[Section 3]{2011arXiv1104.5151D},
for every $f\in C_{c}(G,A)$, $\pi\rtimes U(f)\in B(X)$ is defined
by 
\[
\pi\rtimes U(f)x:=\int_{G}\pi(f(r))U_{r}x\, dr\quad(x\in X).
\]
The map $\pi\rtimes U:C_{c}(G,A)\to B(X)$ is then a representation
of the algebra $C_{c}(G,A)$ on $X$, and is called the \emph{integrated
form }of $(\pi,U)$.

Let $(A,G,\alpha)$ be a Banach algebra dynamical system and $\mathcal{R}$
a class of continuous covariant representations of $(A,G,\alpha)$
on Banach spaces. We will always tacitly assume that $\mathcal{R}$
is non-empty. We will say $\mathcal{R}$ is a \emph{uniformly bounded
class }of continuous covariant representations if there exists a constant
$C\geq0$ and a function $\nu:G\to\mathbb{R}_{\geq0}$, which is bounded
on compact subsets of $G$, such that, for all $(\pi,U)\in\mathcal{R}$,
$\|\pi\|\leq C$ and $\|U_{r}\|\leq\nu(r)$ for all $r\in G$. 

Let $(A,G,\alpha)$ be a Banach algebra dynamical system and $\mathcal{R}$
a uniformly bounded class of continuous covariant representations
of $(A,G,\alpha)$ on Banach spaces. It follows that $\|\pi\rtimes U(f)\|\leq C\left(\sup_{r\in\mbox{supp}(f)}\nu(r)\right)\|f\|_{1}$
for all $(\pi,U)\in\mathcal{R}$ and $f\in C_{c}(G,A)$ \cite[Remark 3.3]{2011arXiv1104.5151D}.
We introduce the (hence finite) algebra seminorm $\sigma^{\mathcal{R}}$
on $C_{c}(G,A)$, defined by 
\[
\sigma^{\mathcal{R}}(f):=\sup_{(\pi,U)\in\mathcal{R}}\|\pi\rtimes U(f)\|\quad(f\in C_{c}(G,A)).
\]
The kernel of $\sigma^{\mathcal{R}}$ is a two-sided ideal of $C_{c}(G,A)$,
hence $C_{c}(G,A)/\ker\sigma^{\mathcal{R}}$ is a normed algebra with
norm $\|\cdot\|^{\mathcal{R}}$ induced by $\sigma^{\mathcal{R}}$.
Its completion is called the \emph{crossed product }(\emph{associated
with $(A,G,\alpha)$ and $\mathcal{R}$}), and denoted by $\crossedprod$.
Multiplication in $\crossedprod$ will be denoted by $*$. We denote
the quotient map from $C_{c}(G,A)$ to $\crossedprod$ by $q^{\mathcal{R}}:C_{c}(G,A)\to\crossedprod$.
For any Banach space $X$ and linear map $T:C_{c}(G,A)\to X$, if
$T$ is bounded with respect to the $\sigma^{\mathcal{R}}$ seminorm,
we will denote the canonically induced bounded linear map on $\crossedprod$
by $T^{\mathcal{R}}:\crossedprod\to X$, as detailed in \cite[Section 3]{2011arXiv1104.5151D}.

\subsection{\label{sub:classical-Correspondence}Correspondence between representations
of $(A,G,\alpha)$ and $\crossedprod$}

We briefly describe the General Correspondence Theorem \cite[Theorem 8.1]{2011arXiv1104.5151D},
most of which is formulated as Theorem \ref{thm:General-Correspondence-Theorem}
below. In the presence of a bounded approximate left identity of $A$,
the General Correspondence Theorem describes a bijection between the
non-degenerate $\mathcal{R}$-continuous (to be defined below) covariant
representations of a Banach algebra dynamical system $(A,G,\alpha)$
and the non-degenerate bounded representations of the associated crossed
product $\crossedprod$. In Section \ref{sec:Pre-ordered-Banach-algebra-dynamical-systems-and-crossed-prods}
we will adapt the results from this section to pre-ordered Banach
algebra dynamical systems and the associated pre-ordered crossed products.
\begin{defn}
Let $(A,G,\alpha)$ be a Banach algebra dynamical system and $\mathcal{R}$
a uniformly bounded class of continuous covariant representations
of $(A,G,\alpha)$ on Banach spaces. If $(\pi,U)$ is a continuous
covariant representation of $(A,G,\alpha)$ on a Banach space $X$,
and $\pi\rtimes U:C_{c}(G,A)\to B(X)$ is bounded with respect to
$\sigma^{\mathcal{R}}$, we will say $(\pi,U)$ is $\mathcal{R}$\emph{-continuous}.
\end{defn}
The proof of the General Correspondence Theorem proceeds through an
application of Theorem \ref{thm:extension-theorem}, which requires
the existence of a bounded approximate left identity of $\crossedprod$.
The following theorem makes precise how this (and its right-sided
counterpart) is inherited from $A$.
\begin{thm}
\label{thm:bounded-approx-id-incrossed-prod}\cite[Theorem 4.4]{2011arXiv1104.5151D}
Let $(A,G,\alpha)$ be a Banach algebra dynamical system, and let
$\mathcal{R}$ be a uniformly bounded class of continuous covariant
representations of $(A,G,\alpha)$ on Banach spaces. Let $A$ have
a bounded approximate left \textup{(}right\textup{)} identity $(u_{i})$.
Let $\mathcal{Z}$ be a neighbourhood base of $e\in G$ of which all
elements are contained in a fixed compact subset of $G$. For each
$V\in\mathcal{Z}$, let $z_{V}\in C_{c}(G)$ be positive, supported
in $V$, and have integral equal to one. Then the net 
\[
\left(q^{\mathcal{R}}(z_{V}\otimes u_{i})\right),
\]
where $(V,i)\leq(W,j)$ is defined to mean $i\leq j$ and $W\subseteq V$,
is a bounded approximate left \textup{(}right\textup{)} identity of
$\crossedprod$.
\end{thm}
Let $(A,G,\alpha)$ be a Banach algebra dynamical system. We define
the maps $i_{A}:A\to\mbox{End}(C_{c}(G,A))$ and $i_{G}:G\to\mbox{End}(C_{c}(G,A))$
by 
\begin{eqnarray*}
(i_{A}(a)f)(s) & := & af(s),\\
(i_{G}(r)f)(s) & := & \alpha_{r}(f(r^{-1}s)),
\end{eqnarray*}
for $f\in C_{c}(G,A)$, $a\in A$ and $r,s\in G$. The maps $i_{A}(a)$
and $i_{G}(r)$ are bounded on $C_{c}(G,A)$ with respect to $\sigma^{\mathcal{R}}$
\cite[Lemma 6.3]{2011arXiv1104.5151D}, hence we can define the maps
$i_{A}^{\mathcal{R}}:A\to B(\crossedprod)$ and $i_{G}^{\mathcal{R}}:G\to B(\crossedprod)$,
by $i_{A}^{\mathcal{R}}(a):=i_{A}(a)^{\mathcal{R}}$ and $i_{G}^{\mathcal{R}}(r):=i_{G}(r)^{\mathcal{R}}$
in the notation of Section \ref{sub:unordered-basic-definitions},
for $a\in A$ and $r\in G$. Moreover, the maps $a\mapsto i_{A}^{\mathcal{R}}(a)$
and $r\mapsto i_{G}^{\mathcal{R}}(r)$ map $A$ and $G$ into $\leftcent(\crossedprod)$
\cite[Proposition 6.4]{2011arXiv1104.5151D}. If $A$ has a bounded
approximate left  identity and $\mathcal{R}$ is a uniformly bounded
class of non-degenerate continuous covariant representations of $(A,G,\alpha)$
on Banach spaces, then $(i_{A}^{\mathcal{R}},i_{G}^{\mathcal{R}})$
is a non-degenerate $\mathcal{R}$-continuous covariant representation
of $(A,G,\alpha)$ on $\crossedprod$, and the integrated form $(i_{A}^{\mathcal{R}}\rtimes i_{G}^{\mathcal{R}})^{\mathcal{R}}$
equals the left regular representation of $\crossedprod$ \cite[Theorem 7.2]{2011arXiv1104.5151D}. 

This pair $(i_{A}^{\mathcal{R}},i_{G}^{\mathcal{R}})$ can be used
to ``generate'' non-degenerate continuous covariant representations
of $(A,G,\alpha)$ from non-degenerate bounded representations of
$\crossedprod$. We will investigate this further in Section \ref{sec:Uniqueness-of-the-crossed-prod},
but its key role becomes already apparent in the following result,
giving an explicit bijection between the non-degenerate $\mathcal{R}$-continuous
covariant representations of $(A,G,\alpha)$ and the non-degenerate
bounded representations of $\crossedprod$.
\begin{thm}
\textup{(}General Correspondence Theorem, cf.\,\cite[Theorem 8.1]{2011arXiv1104.5151D}\textup{)}
\label{thm:General-Correspondence-Theorem}Let the triple $(A,G,\alpha)$ be
a Banach algebra dynamical system, where $A$ has a bounded approximate
left identity. Let $\mathcal{R}$ be a uniformly bounded class of
non-degenerate continuous covariant representations of $(A,G,\alpha)$
on Banach spaces. Then the map $(\pi,U)\mapsto(\pi\rtimes U)^{\mathcal{R}}$
is a bijection between the non-degenerate $\mathcal{R}$-continuous
covariant representations of $(A,G,\alpha)$ on Banach spaces and
the non-degenerate bounded representations of $\crossedprod$ on such
spaces.

More precisely:
\begin{enumerate}
\item If the pair $(\pi,U)$ is a non-degenerate $\mathcal{R}$-continuous covariant
representation of $(A,G,\alpha)$ on a Banach space $X_{\pi}$, then
$(\pi\rtimes U)^{\mathcal{R}}$ is a non-degenerate bounded representation
of $\crossedprod$ on $X_{\pi}$, and 
\[
(\overline{(\pi\rtimes U)^{\mathcal{R}}}\circ i_{A}^{\mathcal{R}},\overline{(\pi\rtimes U)^{\mathcal{R}}}\circ i_{G}^{\mathcal{R}})=(\pi,U),
\]
where $\overline{(\pi\rtimes U)^{\mathcal{R}}}$ is the non-degenerate
bounded representation of \textup{$\leftcent(\crossedprod)$ }as in
Theorem \ref{thm:extension-theorem}\textup{.} 
\item If $T$ is a non-degenerate bounded representation of $\crossedprod$
on a Banach space $X_{T}$, then $(\overline{T}\circ i_{A}^{\mathcal{R}},\overline{T}\circ i_{G}^{\mathcal{R}})$
is a non-degenerate $\mathcal{R}$-continuous covariant representation
of $(A,G,\alpha)$ on $X_{T}$, and
\[
(\overline{T}\circ i_{A}^{\mathcal{R}}\rtimes\overline{T}\circ i_{G}^{\mathcal{R}})^{\mathcal{R}}=T
\]
where $\overline{T}$ is the non-degenerate bounded representation
of \textup{$\leftcent(\crossedprod)$ }as in Theorem \ref{thm:extension-theorem}\textup{.} 
\end{enumerate}
\end{thm}

\section{Pre-ordered Banach algebra dynamical systems and crossed products\label{sec:Pre-ordered-Banach-algebra-dynamical-systems-and-crossed-prods}}

In this section we study pre-ordered Banach algebra dynamical systems
and their associated crossed products. After the preliminary Section
\ref{sub:Pre-ordered-Banach-algebra-dynamical-systems}, in Section
\ref{sub:Crossed-products-of-pre-ordered-banach-alg} we will define
pre-ordered crossed products associated with pre-ordered Banach algebra
dynamical systems, and describe properties of the cone of such pre-ordered
crossed products (cf.\,Theorem \ref{thm:ordered-crossed-prod-properties}).
Finally, Theorem \ref{thm:Ordered-general-correspondence} in Section
\ref{sub:Correspondence-between-positive-representations} will give
an adaptation of the General Correspondence Theorem (Theorem \ref{thm:General-Correspondence-Theorem})
to the ordered context.

\subsection{Pre-ordered Banach algebra dynamical systems\label{sub:Pre-ordered-Banach-algebra-dynamical-systems}}

We introduce pre-ordered Banach algebra dynamical systems $(A,G,\alpha)$,
and verify that the twisted convolution as defined in Section \ref{sub:unordered-basic-definitions}
gives $C_{c}(G,A)$ a pre-ordered algebra structure. Furthermore,
Lemma \ref{lem:integrated-froms-positive} shows that positive continuous
covariant representations of a pre-ordered Banach algebra dynamical
systems $(A,G,\alpha)$ have positive integrated forms.
\begin{defn}
Let $A$ be a pre-ordered Banach algebra with closed cone, $G$ a
locally compact group, and $\alpha:G\to\mbox{Aut}_{+}(A)$ a strongly
continuous positive representation of $G$ on $A$. Then we will call
the triple $(A,G,\alpha)$ a \emph{pre-ordered} \emph{Banach algebra
dynamical system. }\end{defn}
\begin{lem}
\label{lemma:twisted-convolution-positive}If $(A,G,\alpha)$ is a
pre-ordered Banach algebra dynamical system, with $A$ having a closed
cone, then $(C_{c}(G,A),C_{c}(G,A_{+}))$, with twisted convolution
as defined in Section \ref{sub:unordered-basic-definitions}, is a
pre-ordered algebra.\end{lem}
\begin{proof}
Let $f,g\in C_{c}(G,A_{+})$ with $f\neq0$. By \cite[Theorem 3.27]{Rudin},
for every $s\in G$, 
\[
\frac{(f*g)(s)}{\mu(\textup{supp}(f))}=\int_{\textup{supp}(f)}f(r)\alpha_{r}(g(r^{-1}s))\,\frac{dr}{\mu(\textup{supp}(f))},
\]
where $\mu$ denotes the chosen left Haar measure on $G$, lies in
the closed convex hull of $\{f(r)\alpha_{r}(g(r^{-1}s)):r\in\textup{supp}(f)\}\subseteq A_{+}$.
Since $A_{+}$ is itself closed and convex, the result follows.\end{proof}
\begin{defn}
Let $(A,G,\alpha)$ be a pre-ordered Banach algebra dynamical system.
Let $(\pi,U)$ be a covariant representation of $(A,G,\alpha)$ on
a pre-ordered normed space $X$. If both $\pi$ and $U$ are positive
representations, we will say that the covariant representation $(\pi,U)$
is \emph{positive}.\end{defn}
\begin{lem}
\label{lem:integrated-froms-positive}Let $(A,G,\alpha)$ be a pre-ordered
Banach algebra dynamical system with $A$ having a closed cone, and
$(\pi,U)$ a positive continuous covariant representation of $(A,G,\alpha)$
on a pre-ordered Banach space $X$ with a closed cone. Then the integrated
form $\pi\rtimes U:C_{c}(G,A)\to B(X)$ is a positive algebra representation.\end{lem}
\begin{proof}
Let $f\in C_{c}(G,A_{+})$. Since $(\pi,U)$ is positive, we have
$\pi(f(r))U_{r}x\in X_{+}$ for all $x\in X_{+}$ and $r\in G$. Since
$X_{+}$ is closed and convex, we obtain $\int_{G}\pi(f(r))U_{r}x\, dr\in X_{+}$
as in the proof of Lemma \ref{lemma:twisted-convolution-positive}. 
\end{proof}

\subsection{Crossed products associated with pre-ordered Banach algebra dynamical
systems\label{sub:Crossed-products-of-pre-ordered-banach-alg}}

In this section we will describe the construction of pre-ordered crossed
products associated with pre-ordered Banach algebra dynamical systems.
The construction as a Banach algebra is as described in Section \ref{sub:unordered-basic-definitions},
so we will focus mainly on the properties of the order structure.
\begin{lem}
\label{lem:Cc(G,A)-quotient-is-ordered-normed-algebra}Let $(A,G,\alpha)$
be a pre-ordered Banach algebra dynamical system, with $A$ having
a closed cone, and $\mathcal{R}$ a uniformly bounded class of continuous
covariant representations of $(A,G,\alpha)$ on Banach spaces. Then
the space $C_{c}(G,A)/\ker\sigma^{\mathcal{R}}$, with norm $\|\cdot\|^{\mathcal{R}}$
induced by $\sigma^{\mathcal{R}}$ and pre-ordered by the quotient
cone $q^{\mathcal{R}}(C_{c}(G,A_{+}))$, is a pre-ordered normed algebra.\end{lem}
\begin{proof}
As explained in Section \ref{sub:unordered-basic-definitions}, $C_{c}(G,A)/\ker\sigma^{\mathcal{R}}$
is a normed algebra with norm induced by $\sigma^{\mathcal{R}}$.
That it is a pre-ordered algebra follows from the definition of the
quotient cone and the fact that the twisted convolution of positive
elements of $C_{c}(G,A)$ is again positive by Lemma \ref{lemma:twisted-convolution-positive}.
\end{proof}
We can now describe $\crossedprod$ as a pre-ordered Banach algebra: 
\begin{defn}
\label{def:pre-ordered-crossed-prod}Let $(A,G,\alpha)$ be a pre-ordered
Banach algebra dynamical system, with $A$ having a closed cone, and
$\mathcal{R}$ a uniformly bounded class of continuous covariant representations
of $(A,G,\alpha)$ on Banach spaces. The completion  of the
pre-ordered normed algebra $(C_{c}(G,A)/\ker\sigma^{\mathcal{R}},q^{\mathcal{R}}(C_{c}(G,A_{+})))$ (in the sense
of Definition \ref{def:completion-or-ordered-normed-space}),
with norm $\|\cdot\|^{\mathcal{R}}$ induced by $\sigma^{\mathcal{R}}$,
will be denoted by $\crossedprod$, the pre-ordering being tacitly
understood, and will be called the\emph{ pre-ordered crossed product}
(\emph{associated with }$(A,G,\alpha)$\emph{ and $\mathcal{R}$}).
\end{defn}
We recall the following result from \cite{2011arXiv1104.5151D}, which
will be used twice in the proof of Theorem \ref{thm:ordered-crossed-prod-properties}:
\begin{prop}
\label{prop:norm-in-(AxG)R}\cite[Proposition 3.4]{2011arXiv1104.5151D}
Let $(A,G,\alpha)$ be a Banach algebra dynamical system, and let
$\mathcal{R}$ be a uniformly bounded class of continuous covariant
representations of $(A,G,\alpha)$ on Banach spaces. Then, for every
$d\in\crossedprod$, 
\[
\|d\|^{\mathcal{R}}=\sup_{(\pi,U)\in\mathcal{R}}\|(\pi\rtimes U)^{\mathcal{R}}(d)\|.
\]

\end{prop}
The following theorem describes properties of the closed cone $\crossedprod_{+}$
in a pre-ordered crossed product $\crossedprod$. 
\begin{thm}
\label{thm:ordered-crossed-prod-properties}Let $(A,G,\alpha)$ be
a pre-ordered Banach algebra dynamical system with $A$ having a closed
cone. Let $\mathcal{R}$ be a uniformly bounded class of continuous
covariant representations of $(A,G,\alpha)$ on Banach spaces. Then:
\begin{enumerate}
\item The pre-ordered crossed product $\crossedprod$ is a pre-ordered Banach
algebra with a closed cone.
\item If $A_{+}$ is generating in $A$, then $\crossedprod_{+}$ is topologically
generating in $\crossedprod$. 
\item If $A_{+}$ is generating in $A$ and $(\cdot)^{+}:q^{\mathcal{R}}(C_{c}(G,A))\to q^{\mathcal{R}}(C_{c}(G,A_{+}))$
is a function such that $q^{\mathcal{R}}(f)\leq q^{\mathcal{R}}(f)^{+}$
for all $f\in C_{c}(G,A)$, and which maps $\|\cdot\|^{\mathcal{R}}$-Cauchy
sequences to $\|\cdot\|^{\mathcal{R}}$-Cauchy sequences, then the
cone $\crossedprod_{+}$ is generating in $\crossedprod$.
\item If $\mathcal{R}$ is a uniformly bounded class of positive continuous
covariant representations of $(A,G,\alpha)$ on pre-ordered Banach
spaces with closed cones such that, for every $(\pi,U)\in\mathcal{R}$,
the cone $B(X_{\pi})_{+}$ is proper, then the cone $\crossedprod_{+}$
is a proper cone.
\item If $\beta>0$ and $\mathcal{R}$ is a uniformly bounded class of positive
continuous covariant representations of $(A,G,\alpha)$ on pre-ordered
Banach spaces with closed cones such that, for every $(\pi,U)\in\mathcal{R}$,
$B(X_{\pi})$ is $\beta$-\textup{(}absolutely\textup{)} normal,
then $\crossedprod$ is $\beta$-\textup{(}absolutely\textup{)} normal.
\end{enumerate}
\end{thm}
\begin{proof}
As to (1): That $\crossedprod$ is a pre-ordered Banach algebra with
closed cone is immediate from Lemmas \ref{lem:Cc(G,A)-quotient-is-ordered-normed-algebra}
and \ref{lem:completion-of-normed-algebra-isbanach-algebra}.

We prove (2). Let $A_{+}$ be generating in $A$. By Corollary \ref{thm:C(Omega,X+)-generating-if-X+-is-generating},
the cone $C_{c}(G,A_{+})$ is generating in $C_{c}(G,A)$. Hence
the quotient cone is generating in $C_{c}(G,A)/\ker\sigma^{\mathcal{R}}$,
and by  Lemma \ref{lem:generating-becomes-weakly-generating-in-completions},
$\crossedprod_{+}$ is topologically generating in $\crossedprod$. 

The statement in (3) follows from Lemma \ref{lem:generating-with-continuous-pos-part-function-becomes-generating-in-completions}.

We prove (4). Let $d\in\crossedprod$ be such that $0\leq d\leq0$
in $\crossedprod$. Let $(\pi,U)\in\mathcal{R}$ be arbitrary. By
Lemma \ref{lem:integrated-froms-positive}, $\pi\rtimes U:C_{c}(G,A)\to B(X_{\pi})$
is a positive algebra representation. Therefore the induced map $(\pi\rtimes U)^{\mathcal{R}}:C_{c}(G,A)/\ker\sigma^{\mathcal{R}}\to B(X_{\pi})$
is a positive bounded algebra representation. Since the cone of $X_{\pi}$
is closed, so is the cone of $B(X_{\pi})$, and therefore, by Lemma
\ref{lem:bounded-positive-operator-on-ordered-noremed-space-implies-extension-is-positive},
$(\pi\rtimes U)^{\mathcal{R}}:\crossedprod\to B(X_{\pi})$ is a positive
algebra representation. Hence $0\leq d\leq0$ implies 
\[
0\leq(\pi\rtimes U)^{\mathcal{R}}(d)\leq0,
\]
and since $B(X_{\pi})_{+}$ is a proper cone, we obtain $(\pi\rtimes U)^{\mathcal{R}}(d)=0$.
Therefore, by Proposition \ref{prop:norm-in-(AxG)R}, $\|d\|^{\mathcal{R}}=\sup_{(\pi,U)\in\mathcal{R}}\|(\pi\rtimes U)^{\mathcal{R}}(d)\|=0$,
and hence $d=0$. We conclude that $\crossedprod_{+}$ is a proper
cone.

We prove (5). Let $\beta>0$ be such that, for every $(\pi,U)\in\mathcal{R}$,
$B(X_{\pi})$ is $\beta$-absolutely normal. For any $(\pi,U)\in\mathcal{R}$,
as previously, $(\pi\rtimes U)^{\mathcal{R}}:\crossedprod\to B(X_{\pi})$
is a positive algebra representation. Hence, if $\pm d_{1}\leq d_{2}$
for $d_{1},d_{2}\in\crossedprod$, we have 
\[
\pm(\pi\rtimes U)^{\mathcal{R}}(d_{1})\leq(\pi\rtimes U)^{\mathcal{R}}(d_{2}).
\]
Since $B(X_{\pi})$ is $\beta$-absolutely normal, we obtain $\|(\pi\rtimes U)^{\mathcal{R}}(d_{1})\|\leq\beta\|(\pi\rtimes U)^{\mathcal{R}}(d_{2})\|$.
By Proposition \ref{prop:norm-in-(AxG)R}, taking the supremum over
$\mathcal{R}$ on both sides yields $\|d_{1}\|^{\mathcal{R}}\leq\beta\|d_{2}\|^{\mathcal{R}}$.

That $\crossedprod$ is $\beta$-normal for some $\beta>0$ under
the assumption that, for every $(\pi,U)\in\mathcal{R}$, $B(X_{\pi})$
is $\beta$-normal follows similarly. 
\end{proof}
Under the hypotheses of Theorem \ref{thm:ordered-crossed-prod-properties},
we see that it is relatively easy to have a topologically generating
cone of $\crossedprod$: It is sufficient that $A_{+}$ is generating
in $A$. The condition in (3) under which the cone of $\crossedprod$
is generating is less easily verified if $\sigma^{\mathcal{R}}$ is
not a norm, but we will nevertheless see an example (where $\sigma^{\mathcal{R}}$
is a norm) in Section \ref{sec:Pre-ordered-generalized-Beurling}
where we conclude that the cone of $\crossedprod$ is generating through
a different method than provided by (3) in the theorem above. Furthermore,
according to part (4) and Theorem \ref{thm-yamamuro-results}, if
every continuous covariant representation from $\mathcal{R}$ is positive
and acts on a pre-ordered Banach space with a closed proper generating
cone, then the cone of $\crossedprod$ is proper. As to (5), an appeal
to Remark \ref{remark:onyamamuroresults} shows that $\crossedprod$
is $1$-absolutely normal if every continuous covariant representation
from $\mathcal{R}$ is positive and acts on a Banach lattice. We collect
the features of the latter case in the following result:
\begin{cor}
\label{cor:ordered-crossed-prod-properties-Banach-lattice-version}Let
$(A,G,\alpha)$ be a pre-ordered Banach algebra dynamical system with
$A$ having a closed cone. Let $\mathcal{R}$ be a uniformly bounded
class of positive continuous covariant representations on Banach lattices.
Then $\crossedprod_{+}$ is a closed proper cone and $\crossedprod$
is a $1$-absolutely normal ordered Banach algebra, i.e., for $d_{1},d_{2}\in\crossedprod$,
if $\pm d_{1}\leq d_{2}$, then $\|d_{1}\|^{\mathcal{R}}\leq\|d_{2}\|^{\mathcal{R}}$.
If $A_{+}$ is generating in $A$, then $\crossedprod_{+}$ is topologically
generating in $\crossedprod$.
\end{cor}
The following example shows that even with $A$ a Banach lattice algebra
and the positive representations from $\mathcal{R}$ acting on Banach
lattices, $\ker\sigma^{\mathcal{R}}$ need not be an order ideal in
the vector lattice $C_{c}(G,A)$ in general. 
\begin{example}
Let $\mathbb{Z}_{2}:=\mathbb{Z}/2\mathbb{Z}$. We consider the pre-ordered
Banach algebra dynamical system $(\mathbb{R},\mathbb{Z}_{2},\textup{triv})$
and $\mathcal{R}=\{(\textup{id},\mbox{triv})\}$ with $(\textup{id},\mbox{triv})$
the trivial positive non-degenerate continuous covariant representation
of $(\mathbb{R},\mathbb{Z}_{2},\textup{triv})$ on $\mathbb{R}$.
Then, for $f\in C_{c}(\mathbb{Z}_{2})$, $\textup{id}\rtimes\mbox{triv}(f)=f(0)+f(1)$,
hence $f\in\ker\sigma^{\mathcal{R}}$ if and only if $f(0)=-f(1)$.
In particular, since $f\in\ker\sigma^{\mathcal{R}}$ does not imply
$|f|\in\ker\sigma^{\mathcal{R}}$, $\ker\sigma^{\mathcal{R}}$ is
not an order ideal in the vector lattice $C_{c}(\mathbb{Z}_{2})$. 
\end{example}

\subsection{Correspondence between positive representations of $(A,G,\alpha)$
and $\crossedprod$\label{sub:Correspondence-between-positive-representations}}

In this section we give an adaptation of the General Correspondence
Theorem (Theorem \ref{thm:General-Correspondence-Theorem}) to the
ordered context. As in the unordered context, Theorem \ref{thm:extension-theorem}
will be a crucial ingredient, which here will rely on the existence
of a positive bounded approximate left identity of the pre-ordered
crossed product $\crossedprod$. The following result shows that this
is inherited from $A$.
\begin{prop}
\label{prop:ordered-crossed-prod-has-positive-bounded-approx-id}Let
$(A,G,\alpha)$ be a pre-ordered Banach algebra dynamical system with
$A$ having a closed cone, and let $\mathcal{R}$ be a uniformly bounded
class of continuous covariant representations of $(A,G,\alpha)$ on
Banach spaces. Let $A$ have a positive bounded approximate left \textup{(}right\textup{)}
identity $(u_{i})$. Then the net 
\[
\left(q^{\mathcal{R}}(z_{V}\otimes u_{i})\right),
\]
as described in Theorem \ref{thm:bounded-approx-id-incrossed-prod},
is a positive bounded approximate left \textup{(}right\textup{)} identity
of $\crossedprod$.\end{prop}
\begin{proof}
Since the quotient map $q^{\mathcal{R}}:C_{c}(G,A)\to\crossedprod$
is positive and $z_{V}\otimes u_{i}\in C_{c}(G,A_{+})$ for all $i$
and $V\in\mathcal{Z}$, we have $q^{\mathcal{R}}(z_{V}\otimes u_{i})\in\crossedprod_{+}$.
That $\left(q^{\mathcal{R}}(z_{V}\otimes u_{i})\right)$ is a bounded
left (right) identity is the statement of Theorem \ref{thm:bounded-approx-id-incrossed-prod}.
\end{proof}
The following will be used in the proof of Theorem \ref{thm:Ordered-general-correspondence}
and in Section \ref{sec:Uniqueness-of-the-crossed-prod}.
\begin{lem}
\label{lem:iAiG-is-positive}Let $(A,G,\alpha)$ be a pre-ordered
Banach algebra dynamical system, with $A$ having a closed cone. With
$C_{c}(G,A)$ pre-ordered by the cone $C_{c}(G,A_{+})$, the representations
$i_{A}:A\to\textup{End}(C_{c}(G,A))$ and $i_{G}:G\to\textup{End}(C_{c}(G,A))$
as in defined in Section \ref{sub:classical-Correspondence} are positive.

If $\mathcal{R}$ is a uniformly bounded class of non-degenerate continuous
covariant representations of $(A,G,\alpha)$ on Banach spaces, and
$A$ has a bounded approximate left identity, then the pair $(i_{A}^{\mathcal{R}},i_{G}^{\mathcal{R}})$
as defined in Section \ref{sub:classical-Correspondence} is a positive
non-degenerate $\mathcal{R}$-continuous covariant representation
of $(A,G,\alpha)$ on $\crossedprod$ such that $i_{A}^{\mathcal{R}}(A),i_{G}^{\mathcal{R}}(G)\subseteq\leftcent(\crossedprod)$.\end{lem}
\begin{proof}
That the maps $i_{A}:A\to\mbox{End}(C_{c}(G,A))$ and $i_{G}:G\to\mbox{End}(C_{c}(G,A))$
are positive is clear. By \cite[Lemma 6.3]{2011arXiv1104.5151D} and
Lemma \ref{lem:bounded-positive-operator-on-ordered-noremed-space-implies-extension-is-positive}
the operators $i_{A}^{\mathcal{R}}(a)$ and $i_{G}^{\mathcal{R}}(r)$
($a\in A_{+}$, $r\in G$) on $\crossedprod$ are positive. The remaining
statement is contained in \cite[Theorem 7.2]{2011arXiv1104.5151D}.
\end{proof}
We finally establish the following adaptation of the General Correspondence
Theorem to the ordered context. Note that the class $\mathcal{R}$
is not required to consist of positive continuous covariant representations.
Conditions in that vein affect the properties of the cone in $\crossedprod$
(cf. Theorem \ref{thm:ordered-crossed-prod-properties}, Corollary
\ref{cor:ordered-crossed-prod-properties-Banach-lattice-version}),
but are not necessary for the correspondence.
\begin{thm}
\label{thm:Ordered-general-correspondence}Let $(A,G,\alpha)$ be
a pre-ordered Banach algebra dynamical system, with $A$ having a
closed cone and a positive bounded approximate left identity. Let
$\mathcal{R}$ be a uniformly bounded class of non-degenerate continuous
covariant representations of $(A,G,\alpha)$ on Banach spaces. Then
the map $(\pi,U)\mapsto(\pi\rtimes U)^{\mathcal{R}}$ is a bijection
between the positive non-degenerate $\mathcal{R}$-continuous covariant
representations of $(A,G,\alpha)$ on pre-ordered Banach spaces with
closed cones and the positive non-degenerate bounded representations
of $\crossedprod$ on such spaces.

More precisely:
\begin{enumerate}
\item If $(\pi,U)$ is a positive non-degenerate $\mathcal{R}$-continuous
covariant representation of $(A,G,\alpha)$ on a pre-ordered Banach
space $X_{\pi}$ with a closed cone, then $(\pi\rtimes U)^{\mathcal{R}}$
is a positive non-degenerate bounded representation of $\crossedprod$
on $X_{\pi}$, and 
\[
(\overline{(\pi\rtimes U)^{\mathcal{R}}}\circ i_{A}^{\mathcal{R}},\overline{(\pi\rtimes U)^{\mathcal{R}}}\circ i_{G}^{\mathcal{R}})=(\pi,U),
\]
where $\overline{(\pi\rtimes U)^{\mathcal{R}}}$ is the positive non-degenerate
bounded representation of \textup{$\leftcent(\crossedprod)$ }as in
Theorem \ref{thm:extension-theorem}\textup{.} 
\item If $T$ is a positive non-degenerate bounded representation of $\crossedprod$
on a pre-ordered Banach space $X_{T}$ with a closed cone, then $(\overline{T}\circ i_{A}^{\mathcal{R}},\overline{T}\circ i_{G}^{\mathcal{R}})$
is a positive non-degenerate $\mathcal{R}$-continuous covariant representation
of $(A,G,\alpha)$ on $X_{T}$, and
\[
(\overline{T}\circ i_{A}^{\mathcal{R}}\rtimes\overline{T}\circ i_{G}^{\mathcal{R}})^{\mathcal{R}}=T,
\]
where $\overline{T}$ is the positive non-degenerate bounded representation
of \textup{$\leftcent(\crossedprod)$ }as in Theorem \ref{thm:extension-theorem}\textup{.} 
\end{enumerate}
\end{thm}
\begin{proof}
We prove part (1). If $(\pi,U)$ is a positive non-degenerate $\mathcal{R}$-continuous
covariant representation of $(A,G,\alpha)$, by Lemma \ref{lem:integrated-froms-positive}
and Lemma \ref{lem:bounded-positive-operator-on-ordered-noremed-space-implies-extension-is-positive}
we obtain that $(\pi\rtimes U)^{\mathcal{R}}$ is a positive bounded
representation bounded of $\crossedprod$. That it is non-degenerate
and that $(\overline{(\pi\rtimes U)^{\mathcal{R}}}\circ i_{A}^{\mathcal{R}},\overline{(\pi\rtimes U)^{\mathcal{R}}}\circ i_{G}^{\mathcal{R}})=(\pi,U)$
follows by applying the General Correspondence Theorem (Theorem \ref{thm:General-Correspondence-Theorem}). 

We prove part (2). Since it is assumed that $A$ has a positive bounded
approximate left identity, by Proposition \ref{prop:ordered-crossed-prod-has-positive-bounded-approx-id},
$\crossedprod$ has a positive bounded approximate left identity.
By Theorem \ref{thm:extension-theorem}, $\overline{T}:\leftcent(\crossedprod)\to B(X_{T})$
is a positive representation. By Lemma \ref{lem:iAiG-is-positive},
the maps $i_{A}^{\mathcal{R}}:A\to\leftcent(\crossedprod)$ and $i_{G}^{\mathcal{R}}:G\to\leftcent(\crossedprod)$
are both positive, and therefore $(\overline{T}\circ i_{A}^{\mathcal{R}},\overline{T}\circ i_{G}^{\mathcal{R}})$
is a pair of positive representations of respectively $A$ and $G$
on $X$. The General Correspondence Theorem asserts that $(\overline{T}\circ i_{A}^{\mathcal{R}},\overline{T}\circ i_{G}^{\mathcal{R}})$
is also a non-degenerate $\mathcal{R}$-continuous covariant representation,
and that $(\overline{T}\circ i_{A}^{\mathcal{R}}\rtimes\overline{T}\circ i_{G}^{\mathcal{R}})^{\mathcal{R}}=T$.
\end{proof}

\section{Uniqueness of the pre-ordered crossed product\label{sec:Uniqueness-of-the-crossed-prod}}

Let $(A,G,\alpha)$ be a Banach algebra dynamical system and $\mathcal{R}$
a uniformly bounded class of non-degenerate continuous covariant representations
of $(A,G,\alpha)$ on Banach spaces. In \cite[Theorem 4.4]{CPII}
it was shown (under mild further hypotheses) that the crossed product
$\crossedprod$ is the unique Banach algebra (up to topological isomorphism)
such that the triple $(\crossedprod,i_{A}^{\mathcal{R}},i_{G}^{\mathcal{R}})$
generates all non-degenerate $\mathcal{R}$-continuous covariant representations
of $(A,G,\alpha)$, in the sense that, for every non-degenerate bounded
representation $T$ of $\crossedprod$ on a Banach space $X$, $(\overline{T}\circ i_{A}^{\mathcal{R}},\overline{T}\circ i_{G}^{\mathcal{R}})$
is a non-degenerate $\mathcal{R}$-continuous representation of $(A,G,\alpha)$
on $X$, and that, moreover, all non-degenerate $\mathcal{R}$-continuous
covariant representations of $(A,G,\alpha)$ are obtained in this
way.

In this section we will adapt this to pre-ordered Banach algebra dynamical
systems. If $(A,G,\alpha)$ is a pre-ordered Banach algebra dynamical
system, with $A$ having a closed cone, and $\mathcal{R}$ a uniformly
bounded class of positive non-degenerate continuous covariant representations
of $(A,G,\alpha)$ on pre-ordered Banach spaces with closed cones,
then we will show that (under similar mild hypotheses as in the unordered
case) the pre-ordered crossed product $\crossedprod$ is the unique
pre-ordered Banach algebra (up to bipositive topological isomorphism)
such that the triple $(\crossedprod,i_{A}^{\mathcal{R}},i_{G}^{\mathcal{R}})$
generates all positive non-degenerate $\mathcal{R}$-continuous covariant
representations of $(A,G,\alpha)$ as described above.

We begin with the general framework for generating positive non-degenerate
$\mathcal{R}$-continuous representations from a suitable basic one
as in \cite[Lemma 4.1]{CPII}. 
\begin{lem}
\label{lem:existence_of_generating_pair}Let $(A,G,\alpha)$ be a
pre-ordered Banach algebra dynamical system with $A$ having a closed
cone, and let $\mathcal{R}$ be a uniformly bounded class of continuous
covariant representations of $(A,G,\alpha)$ on Banach spaces. Let
$E$ be a pre-ordered Banach algebra \textup{(}with a not necessarily
closed cone\textup{)} and positive bounded approximate left identity,
and let $(k_{A},k_{G})$ be a positive non-degenerate $\mathcal{R}$-continuous
covariant representation of $(A,G,\alpha)$ on the pre-ordered Banach
space $E$ such that $k_{A}(A),k_{G}(G)\subseteq\leftcent(E)$. Suppose
$T:E\to B(X)$ is a positive non-degenerate bounded representation
of $E$ on a pre-ordered Banach space $X$ with a closed cone. Let
$\overline{T}:\leftcent(E)\to B(X)$ be the positive non-degenerate
bounded representation of $\leftcent(E)$ such that the following
diagram commutes:
\[
\xymatrix{E\ar[dr]^{\lambda}\ar[r]^{T} & B(X)\\
 & \leftcent(E)\ar[u]_{\overline{T}}
}
\]
Then the pair $(\overline{T}\circ k_{A},\overline{T}\circ k_{G})$
is a positive non-degenerate $\mathcal{R}$-continuous covariant representation
of $(A,G,\alpha)$, and $(\overline{T}\circ k_{A})\rtimes(\overline{T}\circ k_{G})=\overline{T}\circ(k_{A}\rtimes k_{G})$.\end{lem}
\begin{proof}
That $(\overline{T}\circ k_{A},\overline{T}\circ k_{G})$ is a non-degenerate
$\mathcal{R}$-continuous covariant representation of $(A,G,\alpha)$
and that $(\overline{T}\circ k_{A})\rtimes(\overline{T}\circ k_{G})=\overline{T}\circ(k_{A}\rtimes k_{G})$
follows from \cite[Lemma 4.1]{CPII}. That $(\overline{T}\circ k_{A},\overline{T}\circ k_{G})$
is positive follows from $(k_{A},k_{G})$ being positive and $\overline{T}:\leftcent(E)\to B(X)$
being positive by Theorem \ref{thm:extension-theorem}.
\end{proof}
Therefore, given a pre-ordered Banach algebra $E$ with such a positive
non-degenerate $\mathcal{R}$-continuous covariant representation
$(k_{A},k_{G})$ of $(A,G,\alpha)$ on $E$, positive non-degenerate
$\mathcal{R}$-continuous covariant representations of $(A,G,\alpha)$
can be generated from positive non-degenerate bounded representations
of $E$.

Clearly, any pre-ordered Banach algebra $E'$ that is bipositively
topologically isomorphic to $E$ must also have a similar positive
non-degenerate $\mathcal{R}$-continuous covariant generating pair
$(k_{A}',k_{G}')$. This is outlined in the following straightforward
adaptation of \cite[Lemma 4.2]{CPII} to the ordered context.
\begin{lem}
\label{lem:generating-pair-translation}Let $(A,G,\alpha)$, $\mathcal{R}$,
$E$ and $(k_{A},k_{G})$ be as in Lemma \ref{lem:existence_of_generating_pair}.
Suppose $E'$ is a pre-ordered Banach algebra and $\psi:E\to E'$
is a bipositive topological isomorphism. Then:
\begin{enumerate}
\item $\psi_{l}:\leftcent(E)\to\leftcent(E')$, defined by $\psi_{l}(L):=\psi L\psi^{-1}$
for $L\in\leftcent(E)$, is a bipositive topological isomorphism.
\item The pair $(k_{A}',k_{G}'):=(\psi_{l}\circ k_{A},\psi_{l}\circ k_{G})$
is a positive non-degenerate $\mathcal{R}$-continuous covariant representation
of $(A,G,\alpha)$ on $E'$ such that $k_{A}'(A)\subseteq\leftcent(E')$ and $k_{G}'(G)\subseteq\leftcent(E')$.
\item If $T:E\to B(X)$ is a positive non-degenerate bounded representation,
then so is $T':E'\to B(X)$, where $T':=T\circ\psi^{-1}$.
\item If $T:E\to B(X)$ is a positive non-degenerate bounded representation
on a pre-ordered Banach space with a closed cone, and $\overline{T'}:\leftcent(E')\to B(X)$
is the positive non-degenerate bounded representation of $\leftcent(E')$
such that the diagram
\[
\xymatrix{E'\ar[dr]^{\lambda}\ar[r]^{T'} & B(X)\\
 & \leftcent(E')\ar[u]_{\overline{T'}}
}
\]
commutes, then $\overline{T}\circ k_{A}=\overline{T'}\circ k_{A}'$
and $\overline{T}\circ k_{G}=\overline{T'}\circ k_{G}'$.
\end{enumerate}
\end{lem}
If $A$ has a positive bounded approximate left identity, then, according
to Proposition \ref{prop:ordered-crossed-prod-has-positive-bounded-approx-id}
and Lemma \ref{lem:iAiG-is-positive}, the triple $(\crossedprod,i_{A}^{\mathcal{R}},i_{G}^{\mathcal{R}})$
satisfies the hypotheses of Lemma \ref{lem:existence_of_generating_pair},
and by Theorem \ref{thm:Ordered-general-correspondence} \emph{all
}positive non-degenerate $\mathcal{R}$-continuous covariant representations
of $(A,G,\alpha)$ can be ``generated'' from positive non-degenerate
bounded representations of $\crossedprod$ as in Lemma \ref{lem:existence_of_generating_pair}.
By Lemma \ref{lem:generating-pair-translation}, a bipositive topological
isomorphism between $\crossedprod$ and another pre-ordered Banach
algebra yields a triple with the same properties. Our aim in the rest
of this section is to establish the converse: If $(E,k_{A},k_{G})$
(where now $E$ has a closed cone) is a ``generating triple'' for
\emph{all }positive non-degenerate $\mathcal{R}$-continuous covariant
representations of $(A,G,\alpha)$ as in Lemma \ref{lem:existence_of_generating_pair},
then, under mild additional hypotheses, this triple can be obtained
from $(\crossedprod,i_{A}^{\mathcal{R}},i_{G}^{\mathcal{R}})$ via
a bipositive topological isomorphism as in Lemma \ref{lem:generating-pair-translation}
(cf.\,Corollary \ref{cor:sufficient-conditions-uniqueness-of-crossed-prod}).

In order to do this, we will need the existence of a positive isometric
representation of $\crossedprod$ on some pre-ordered Banach space
with a closed cone. As in \cite[Proposition 3.4]{CPII}, this is achieved
through combining sufficiently many members of $\mathcal{R}$ into
one suitable positive continuous covariant representation.
\begin{defn}
\label{def:equivalent_uniformly_bounded_sets}Let $(A,G,\alpha)$
be a Banach algebra dynamical system. Let $\mathcal{R}$ be a uniformly
bounded class of possibly degenerate continuous covariant representations
of $(A,G,\alpha)$ on Banach spaces. We define $[\mathcal{R}]$ to
be the collection of all uniformly bounded classes $S$ that are actually
sets and satisfy $\sigma^{\mathcal{R}}=\sigma^{S}$ on $C_{c}(G,A)$.
Elements of some $[\mathcal{R}]$ will be called \emph{uniformly bounded
sets of continuous covariant representations.}
\end{defn}
We note that $[\mathcal{R}]$ is always non-empty: For every $f\in C_{c}(G,A)$,
considering the set $\{\|\pi\rtimes U(f)\|:(\pi,U)\in\mathcal{R}\}\subseteq\mathbb{R}$
(subclasses of sets are sets), we may choose a sequence from $\{\|\pi\rtimes U(f)\|:(\pi,U)\in\mathcal{R}\}$
converging to $\sigma^{\mathcal{R}}(f)$, and consider only those
corresponding covariant representations from $\mathcal{R}$. In this
way we may choose a set $S$ from $\mathcal{R}$ of cardinality at
most $|C_{c}(G,A)\times\mathbb{N}|$ such that $\sigma^{S}(f)=\sigma^{\mathcal{R}}(f)$
for all $f\in C_{c}(G,A)$. Therefore the previous definition is non-void.

{}
\begin{defn}
Let $I$ be an index set and $\{X_{i}:i\in I\}$ a family of pre-ordered
Banach spaces with closed cones. For $1\leq p\leq\infty$, we will
denote the $\ell^{p}$-direct sum of $\{X_{i}:i\in I\}$ by $\ell^{p}\{X_{i}:i\in I\}$
and endow it with the cone $\ell^{p}\{(X_{i})_{+}:i\in I\}$, so that
it is a pre-ordered Banach space with a closed cone.
\end{defn}
{}
\begin{defn}
\label{def:lp-direct-sum-realization}Let $(A,G,\alpha)$ be a pre-ordered
Banach algebra dynamical system, with $A$ having a closed cone, and
$\mathcal{R}$ a uniformly bounded class of positive continuous covariant
representations of $(A,G,\alpha)$ on pre-ordered Banach spaces with
closed cones. For $S\in[\mathcal{R}]$ and $1\leq p<\infty$, suppressing
the $p$-dependence in the notation, we define the positive representations
$(\oplus_{S}\pi):A\to B(\ell^{p}\{X_{\pi}:(\pi,U)\in S\})$ and $(\oplus_{S}U):G\to B(\ell^{p}\{X_{\pi}:(\pi,U)\in S\})$
by $(\oplus_{S}\pi)(a):=\bigoplus_{(\pi,U)\in S}\pi(a)$ and $(\oplus_{S}U)_{r}:=\bigoplus_{(\pi,U)\in S}U_{r}$,
for all $a\in A$ and $r\in G$ respectively. 
\end{defn}
It is easily seen that $((\oplus_{S}\pi),(\oplus_{S}U))$ is a positive
continuous covariant representation, that 
\[
((\oplus_{S}\pi)\rtimes(\oplus_{S}U))(f)=\bigoplus_{(\pi,U)\in S}\pi\rtimes U(f),
\]
and that $\|((\oplus_{S}\pi)\rtimes(\oplus_{S}U))(f)\|=\sigma^{S}(f)=\sigma^{\mathcal{R}}(f)$,
for all $f\in C_{c}(G,A)$. 

We hence obtain the following (where the statement concerning non-degeneracy
is an elementary verification).
\begin{prop}
\label{prop:isometric_realization}Let $(A,G,\alpha)$ be a pre-ordered
Banach algebra dynamical system, where $A$ has a closed cone, and
$\mathcal{R}$ a uniformly bounded class of positive continuous covariant
representations of $(A,G,\alpha)$ on pre-ordered Banach spaces with
closed cones. For any $S\in[\mathcal{R}]$ and $1\leq p<\infty,$
there exists a positive $\mathcal{R}$-continuous covariant representation
of $(A,G,\alpha)$ on the pre-ordered Banach space $\ell^{p}\{X_{\pi}:(\pi,U)\in S\}$
with a closed cone, denoted $((\oplus_{S}\pi),(\oplus_{S}U))$, such
that its positive integrated form satisfies $\|((\oplus_{S}\pi)\rtimes(\oplus_{S}U))(f)\|=\sigma^{\mathcal{R}}(f)$
for all $f\in C_{c}(G,A)$ and hence induces a positive isometric
representation of $\crossedprod$ on $\ell^{p}\{X_{\pi}:(\pi,U)\in S\}$.

If every element of $S$ is non-degenerate, then $((\oplus_{S}\pi),(\oplus_{S}U))$
is non-degenerate.
\end{prop}
In the following theorem we will give sufficient conditions under which
a triple $(E,k_{A},k_{G})$, generating all positive non-degenerate
$\mathcal{R}$-continuous covariant representations of a pre-ordered
Banach algebra dynamical system $(A,G,\alpha)$ as in Lemma \ref{lem:existence_of_generating_pair},
can be obtained from $(\crossedprod,i_{A}^{\mathcal{R}},i_{G}^{\mathcal{R}})$
through a bipositive topological isomorphism as in Lemma \ref{lem:generating-pair-translation}.
\begin{thm}
\label{thm:universal-property}Let $(A,G,\alpha)$ be a pre-ordered
Banach algebra dynamical system with $A$ having a closed cone and
a positive bounded approximate left identity. Let $\mathcal{R}$ be
a uniformly bounded class of positive non-degenerate continuous covariant
representations of $(A,G,\alpha)$ on pre-ordered Banach spaces with
closed cones. Let $E$ be a pre-ordered Banach algebra, with closed
cone and positive bounded approximate left identity and such that
$\lambda:E\to\lambda(E)\subseteq\leftcent(E)$ is a bipositive topological
embedding. Let $(k_{A},k_{G})$ be a positive non-degenerate $\mathcal{R}$-continuous
covariant representation of $(A,G,\alpha)$ on the pre-ordered Banach
space $E$ such that:
\begin{enumerate}
\item $k_{A}(A),k_{G}(G)\subseteq\leftcent(E)$,
\item $(k_{A}\rtimes k_{G})(C_{c}(G,A))\subseteq\lambda(E)$,
\item $(k_{A}\rtimes k_{G})(C_{c}(G,A))$ is dense in $\lambda(E)$,
\item $(k_{A}\rtimes k_{G})(C_{c}(G,A_{+}))$ is dense in $\lambda(E)\cap\leftcent(E)_{+}$.
\end{enumerate}
Suppose that, for every positive non-degenerate $\mathcal{R}$-continuous
covariant representation $(\pi,U)$ of $(A,G,\alpha)$ on a pre-ordered
Banach space $X$ with a closed cone, there exists a positive non-degenerate
bounded representation $T:E\to B(X)$ such that the positive non-degenerate
bounded representation $\overline{T}:\leftcent(E)\to B(X)$ in the
commuting diagram
\[
\xymatrix{E\ar[dr]^{\lambda}\ar[r]^{T} & B(X)\\
 & \leftcent(E)\ar[u]_{\overline{T}}
}
\]
generates $(\pi,U)$ as in Lemma \ref{lem:existence_of_generating_pair},
i.e., is such that $\overline{T}\circ k_{A}=\pi$ and $\overline{T}\circ k_{G}=U$.

Then there exists a unique topological isomorphism $\psi:\crossedprod\to E$
such that the induced topological isomorphism $\psi_{l}:\leftcent(\crossedprod)\to\leftcent(E)$,
defined by $\psi_{l}(L):=\psi L\psi^{-1}$ for $L\in\leftcent(\crossedprod)$,
induces $(k_{A},k_{G})$ from $(i_{A}^{\mathcal{R}},i_{G}^{\mathcal{R}})$
as in Lemma \ref{lem:generating-pair-translation}, i.e., is such
that $k_{A}=\psi_{l}\circ i_{A}^{\mathcal{R}}$ and $k_{G}=\psi_{l}\circ i_{G}^{\mathcal{R}}$. 

Moreover, $\psi$ is bipositive.
\end{thm}
The proof follows largely as in \cite[Proposition 4.3]{CPII}, but
with some modifications in the first part of the proof, which we now
give.
\begin{proof}
By hypothesis $\mathcal{R}$ consists of positive non-degenerate continuous
covariant representations of $(A,G,\alpha)$ on pre-ordered Banach
spaces with closed cones. Hence Proposition \ref{prop:isometric_realization}
provides a positive non-degenerate $\mathcal{R}$-continuous covariant
representation $(\pi,U)$ of $(A,G,\alpha)$ on a pre-ordered Banach
space $X$ with a closed cone such that $(\pi\rtimes U)^{\mathcal{R}}:\crossedprod\to B(X)$
is a positive non-degenerate isometric representation. By hypothesis,
there exists a positive non-degenerate representation $T:E\to B(X)$
such that $\overline{T}\circ k_{A}=\pi$ and $\overline{T}\circ k_{G}=U$.
By Lemma \ref{lem:existence_of_generating_pair}, we obtain $\pi\rtimes U=(\overline{T}\circ k_{A})\rtimes(\overline{T}\circ k_{G})=\overline{T}\circ(k_{A}\rtimes k_{G})$.
Then, for any $f\in C_{c}(G,A)$, 
\begin{eqnarray*}
\|q^{\mathcal{R}}(f)\| & = & \|(\pi\rtimes U)^{\mathcal{R}}(q^{\mathcal{R}}(f))\|\\
 & = & \|\pi\rtimes U(f)\|\\
 & = & \|\overline{T}\circ(k_{A}\rtimes k_{G})(f)\|\\
 & \leq & \|\overline{T}\|\|k_{A}\rtimes k_{G}(f)\|\\
 & = & \|\overline{T}\|\|(k_{A}\rtimes k_{G})^{\mathcal{R}}(q^{\mathcal{R}}(f))\|.
\end{eqnarray*}
Since $(k_{A},k_{G})$ was assumed to be $\mathcal{R}$-continuous,
we obtain $\|(k_{A}\rtimes k_{G})^{\mathcal{R}}(q^{\mathcal{R}}(f))\|\leq\|(k_{A}\rtimes k_{G})^{\mathcal{R}}\|\|q^{\mathcal{R}}(f)\|$.
Using (2), (3) and the fact that $\lambda(E)$ is closed, it now follows
that $(k_{A}\rtimes k_{G})^{\mathcal{R}}:\crossedprod\to\lambda(E)$
is a topological isomorphism.

Since $(k_{A},k_{G})$ is positive and $\mathcal{R}$-continuous,
and the cone of $E$ is closed, by Lemmas \ref{lem:integrated-froms-positive}
and \ref{lem:bounded-positive-operator-on-ordered-noremed-space-implies-extension-is-positive},
$(k_{A}\rtimes k_{G})^{\mathcal{R}}:\crossedprod\to\lambda(E)$ is
positive. We claim that $(k_{A}\rtimes k_{G})^{\mathcal{R}}$ is bipositive.
Let $b\in E$ be such that $\lambda(b)\in\lambda(E)\cap\leftcent(E)_{+}$,
hence by (4) there exists a sequence $(f_{n})\subseteq C_{c}(G,A_{+})$
such that $(k_{A}\rtimes k_{G})^{\mathcal{R}}(q^{\mathcal{R}}(f_{n}))=(k_{A}\rtimes k_{G})(f_{n})\to\lambda(b)$.
Since $(k_{A}\rtimes k_{G})^{\mathcal{R}}:\crossedprod\to\lambda(E)$
is a topological isomorphism, the sequence $(q^{\mathcal{R}}(f_{n}))\subseteq\crossedprod_{+}$
converges to some $d\in\crossedprod$ and $(k_{A}\rtimes k_{G})^{\mathcal{R}}(d)=\lambda(b)$.
Moreover, since $\crossedprod_{+}$ is closed and $(q^{\mathcal{R}}(f_{n}))\subseteq\crossedprod_{+}$,
we have $d\in\crossedprod_{+}$. We conclude that $(k_{A}\rtimes k_{G})^{\mathcal{R}}$
is bipositive. 

Since $\lambda:E\to\leftcent(E)$ is assumed to be a bipositive topological
embedding, 
\[
\psi:=\lambda^{-1}\circ(k_{A}\rtimes k_{G})^{\mathcal{R}}:\crossedprod\to E
\]
is a bipositive topological isomorphism.

The remainder of the argument proceeds as in the proof of \cite[Theorem 4.4]{CPII}.
\end{proof}
Under the conditions on $(A,G,\alpha)$ and $\mathcal{R}$ as stated
in the previous theorem, one would of course hope that the triple
$(\crossedprod,i_{A}^{\mathcal{R}},i_{G}^{\mathcal{R}})$ automatically
satisfies the hypotheses on $(E,k_{A},k_{G})$, as happens in the
unordered context \cite[Theorem 4.4]{CPII}. Here, as there, the left
regular representation $\lambda:\crossedprod\to\leftcent(\crossedprod)$
is a topological embedding \cite[Proposition 4.3]{CPII}, and, since
$q^{\mathcal{R}}(C_{c}(G,A))$ is dense in $\crossedprod$ and $(i_{A}^{\mathcal{R}}\rtimes i_{G}^{\mathcal{R}})^{\mathcal{R}}=\lambda$
\cite[Theorem 7.2]{2011arXiv1104.5151D}, we have that $(i_{A}^{\mathcal{R}}\rtimes i_{G}^{\mathcal{R}})(C_{c}(G,A))$
is dense in $\lambda(\crossedprod)$. Hence (1), (2) and (3) in Theorem
\ref{thm:universal-property} are satisfied by $(\crossedprod,i_{A}^{\mathcal{R}},i_{G}^{\mathcal{R}})$.
We claim that the additional assumption that $A$ has a positive bounded
approximate right identity gives the remaining conditions that $\lambda:\crossedprod\to\leftcent(\crossedprod)$
is a bipositive topological embedding, and that (4) holds. Indeed,
if this is the case, let $(u_{i})\subseteq\crossedprod_{+}$ be a
positive approximate right identity of $\crossedprod$ (which exists
by Theorem \ref{prop:ordered-crossed-prod-has-positive-bounded-approx-id}),
and let $d\in\crossedprod$ be such that $\lambda(d)\geq0$. Then
$0\leq\lambda(d)u_{i}=d*u_{i}\to d$, so that $d\in\crossedprod_{+}$
. We conclude that $\lambda^{-1}:\lambda(\crossedprod)\to\crossedprod$
is also positive. Since $(i_{A}^{\mathcal{R}}\rtimes i_{G}^{\mathcal{R}})^{\mathcal{R}}=\lambda$,
this also gives that $(i_{A}^{\mathcal{R}}\rtimes i_{G}^{\mathcal{R}})(C_{c}(G,A_{+}))$
is dense in $\lambda(\crossedprod)\cap\leftcent(\crossedprod)_{+}$.
Hence we have the following uniqueness result:
\begin{cor}
\label{cor:sufficient-conditions-uniqueness-of-crossed-prod}Let $(A,G,\alpha)$
be a pre-ordered Banach algebra dynamical system, with $A$ having
a closed cone and both a positive bounded approximate left identity
and a positive bounded approximate right identity. Let $\mathcal{R}$
be a uniformly bounded class of positive non-degenerate continuous
covariant representations of $(A,G,\alpha)$ on pre-ordered Banach
spaces with closed cones. Then $(\crossedprod,i_{A}^{\mathcal{R}},i_{G}^{\mathcal{R}})$
satisfies all hypotheses on the triple $(E,k_{A},k_{G})$ in Theorem
\ref{thm:universal-property}. Hence triples $(E,k_{A},k_{G})$ as
in Theorem \ref{thm:universal-property} exist, and every such ``generating
triple'' for all positive non-degenerate $\mathcal{R}$-continuous
representations of $(A,G,\alpha)$ originates from $(\crossedprod,i_{A}^{\mathcal{R}},i_{G}^{\mathcal{R}})$
through a bipositive topological isomorphism $\psi:\crossedprod\to E$
as in Theorem \ref{thm:universal-property} \textup{(}so that $E$
necessarily has a positive bounded approximate right identity as well\textup{)}.
\end{cor}

\section{Pre-ordered generalized Beurling algebras\label{sec:Pre-ordered-generalized-Beurling}}

In \cite[Section 5]{CPII} it was shown that a generalized Beurling
algebra (to be defined below) is topologically isomorphic to a crossed
product associated with a Banach algebra dynamical system, and the
non-degenerate bounded representations of these algebras were described
in terms of non-degenerate continuous covariant representations of
the underlying Banach algebra dynamical system. We refer the reader
to \cite[Section 5]{CPII} for a more complete treatment of generalized
Beurling algebras and how they are constructed from Banach algebra
dynamical systems.

In this section we will adapt the main results from \cite[Section 5]{CPII}
to the case of pre-ordered Banach algebra dynamical systems and pre-ordered
generalized Beurling algebras. Theorem \ref{thm:Choosing-R-correctly-Crossed-Products-are-beurling}
is the analogue of \cite[Theorem 5.17]{CPII} in the ordered context,
and shows that a pre-ordered generalized Beurling algebra is bipositively
topologically isomorphic to a crossed product associated with a pre-ordered
Banach algebra dynamical system. In Theorem \ref{thm:continuous-non-deg-covars-are-R-continuous}
we modify \cite[Theorem 5.20]{CPII} to explicitly describe a bijection
between the positive non-degenerate continuous covariant representations
of a pre-ordered Banach algebra dynamical system, where the group
representation is bounded by a multiple of a fixed weight on the underlying
group, and the positive non-degenerate bounded representations of
the associated pre-ordered generalized Beurling algebra. 

We begin with a brief description of pre-ordered generalized Beurling
algebras and related spaces.
\begin{defn}
For a locally compact group $G$, let $\omega:G\to[0,\infty)$ be
a non-zero submultiplicative Borel measurable function. Then $\omega$
is called a \emph{weight }on $G$. 
\end{defn}
Note that we do not assume that $\omega\geq1$, as is done in some
parts of the literature. The fact that $\omega$ is non-zero readily
implies that $\omega(e)\geq1$. More generally, if $K\subseteq G$
is a compact set, there exist $a,b>0$ such that $a\leq\omega(s)\leq b$
for all $s\in K$ \cite[Lemma 1.3.3]{Kaniuth}. 

{}
\begin{defn}
\label{def:Beurling-space}Let $X$ be a pre-ordered Banach space
with a closed cone, and $\omega:G\to[0,\infty)$ a weight on $G$.
We define the weighted $1$-norm on $C_{c}(G,X)$ by 
\[
\|h\|_{1,\omega}:=\int_{G}\|h(s)\|\omega(s)\, ds\quad(h\in C_{c}(G,X)),
\]
and define the pre-ordered Banach space $L^{1}(G,X,\omega)$ as the
completion (in the sense of Definition \ref{def:completion-or-ordered-normed-space})
of the pre-ordered vector space $(C_{c}(G,X),C_{c}(G,X_{+}))$ with
the $\|\cdot\|_{1,\omega}$-norm.
\end{defn}
Given the prominent role of continuous compactly supported functions
in the theory, the definition of $L^{1}(G,X,\omega)$ as the completion
of the space $C_{c}(G,X)$ is clearly convenient. A drawback, however,
is that it is then not clear that $L^{1}(G,X,\omega)_{+}$, which
is, by definition, the closure of $C_{c}(G,X_{+})$, is generating
in $L^{1}(G,X,\omega)$ if $X_{+}$ is generating in $X$. From Corollary
\ref{thm:C(Omega,X+)-generating-if-X+-is-generating} we know that
$C_{c}(G,X_{+})$ is generating in $C_{c}(G,X)$, and then Lemma \ref{lem:generating-becomes-weakly-generating-in-completions}
yields that $L^{1}(G,X,\omega)_{+}$ is topologically generating in
$L^{1}(G,X,\omega)$, but generalities do not seem to help us beyond
this point. Similarly, it is not clear that $L^{1}(G,X,\omega)_{+}$
is a proper cone if $X_{+}$ is proper. To establish these results,
we use the fact that, as already observed in \cite[Remark 5.3]{CPII},
$L^{1}(G,X,\omega)$ is isometrically isomorphic to a Bochner space
(also if the left Haar measure $\mu$ is not $\sigma$-finite, or
$X$ is not separable). We recall the relevant facts. A function $f:G\to X$
is Bochner integrable (with respect to $\omega d\mu$) if $f^{-1}(B)$
is a Borel subset of $G$ for every Borel subset $B$ of $X$, $f(G)$
is separable, and $\int_{G}\|f(s)\|\omega(s)\: d\mu(s)<\infty$ (the
measurability of $s\mapsto\|f(s)\|$ is an automatic consequence of
the Borel measurability of $f$). On identifying Bochner integrable
functions that are equal $\omega d\mu$-almost everywhere, one obtains
a Banach space $L^{1}(G,\mathcal{B},\omega d\mu,X)$, where $\mathcal{B}$
is the Borel $\sigma$-algebra of $G$, and the norm is given by $\|[f]\|=\int_{G}\|f(s)\|\omega(s)\, d\mu(s)$,
with $f$ any representative of $[f]\in L^{1}(G,\mathcal{B},\omega d\mu,X)$.
Clearly the inclusion map of $(C_{c}(G,X),\|\cdot\|_{1,\omega})$
into $L^{1}(G,\mathcal{B},\omega d\mu,X)$ is isometric, and the existence
of the aforementioned isometric isomorphism between $L^{1}(G,X,\omega)$
and $L^{1}(G,\mathcal{B},\omega d\mu,X)$ is then established by showing
that $C_{c}(G,X)$ is dense in $L^{1}(G,\mathcal{B},\omega d\mu,X)$.
In the present context, if $X$ is a pre-ordered Banach space, then
$L^{1}(G,\mathcal{B},\omega d\mu,X)$ has a natural cone 
\[
L^{1}(G,\mathcal{B},\omega d\mu,X_{+}):=\{f\in L^{1}(G,\mathcal{B},\omega d\mu,X):f(s)\in X_{+}\textup{ for }\omega d\mu\textup{-a.a. }s\in G\},
\]
where, as usual we have ignored the distinction between equivalence
classes and functions. As in the scalar case, a convergent sequence
in $L^{1}(G,\mathcal{B},\omega d\mu,X)$ has a subsequence that converges
$\omega d\mu$-almost everywhere to the limit function. Hence if $X_{+}$
is closed, then so is $L^{1}(G,\mathcal{B},\omega d\mu,X_{+})$. We
then have the following natural result.
\begin{prop}
\label{prop:Beurlin-and-Bochner-cones-are-the-same}Let $X$ be a
pre-ordered Banach space with a closed cone. Let $G$ be a locally
compact group and $\omega$ a weight on $G$. Then:
\begin{enumerate}
\item The cone $C_{c}(G,X_{+})$ is dense in the closed cone $L^{1}(G,\mathcal{B},\omega d\mu,X_{+})$.
\item $(L^{1}(G,X,\omega),L^{1}(G,X,\omega)_{+})$ and $(L^{1}(G,\mathcal{B},\omega d\mu,X),L^{1}(G,\mathcal{B},\omega d\mu,X_{+}))$
are pre-ordered Banach spaces with closed cones that are bipositively
isometrically isomorphic through an isomorphism that is the identity
on $C_{c}(G,X)$.
\end{enumerate}
\end{prop}
\begin{proof}
For the first part we need, in view of the remarks preceding the proposition,
only show that $C_{c}(G,X_{+})$ is dense in $L^{1}(G,\mathcal{B},\omega d\mu,X_{+})$.
If $f\in L^{1}(G,\mathcal{B},\omega d\mu,X)$, then the proof of \cite[Proposition E.2]{Cohn}
shows that there exists a subset $S$ of $\mathbb{Q}f(G)$ and a sequence
of of simple functions $(f_{n})$, with values in $S$, such that
$f_{n}(s)\to f(s)$ and $\|f_{n}(s)\|\leq\|f(s)\|$ for $\omega d\mu$-almost
every $s\in G$. Hence by the dominated convergence theorem (see the
argument on \cite[p.\,352]{Cohn}) $f_{n}\to f$. An inspection of
the proof of \cite[Proposition E.2]{Cohn} shows that, in fact, $S$
can be chosen to be a subset of $\mathbb{Q}_{\geq0}f(G)$. It is then
clear that the (equivalence classes of) $X_{+}$-valued simple functions
are dense in $L^{1}(G,\mathcal{B},\omega d\mu,X_{+})$. Therefore,
it is sufficient to show that the functions of the form $\chi_{B}\otimes x\in L^{1}(G,\mathcal{B},\omega d\mu,X_{+})$,
where $B\in\mathcal{B}$ and $x\in X_{+}$, can be approximated arbitrarily
closely by elements of $C_{c}(G,X_{+})$. As to this, since $C_{c}(G)$
is dense in the Beurling algebra $L^{1}(G,\omega)$ \cite[Lemma 1.3.5]{Kaniuth},
there exists a sequence $(g_{n})\subseteq C_{c}(G)$ such that $g_{n}\to\chi_{B}$
in $L^{1}(G,\omega)$. Since $\chi_{B}\geq0$ we clearly have $\|g_{n}^{+}\otimes x-\chi_{B}\otimes x\|_{1,\omega}\leq\|g_{n}-\chi_{B}\|_{1,\omega}\|x\|\to0$.
Hence $g_{n}^{+}\otimes x\to\chi_{B}\otimes x$, and the proof is
complete.

The second part is immediate from the first.
\end{proof}
We can now settle the matters mentioned above.
\begin{thm}
\label{thm:bochner-cones-properties}Let $X$ be a pre-ordered Banach
space with a closed cone. Let $G$ be a locally compact group and
$\omega$ a weight on $G$.
\begin{enumerate}
\item If $X_{+}$ is generating in $X$, then the closed cone $L^{1}(G,X,\omega)_{+}$
is generating in $L^{1}(G,X,\omega)$.
\item If $X_{+}$ is a proper cone, then the closed cone $L^{1}(G,X,\omega)_{+}$
is proper.
\end{enumerate}
\end{thm}
\begin{proof}
In view of Proposition \ref{prop:Beurlin-and-Bochner-cones-are-the-same},
it is equivalent to prove the statements for the closed cone $L^{1}(G,\mathcal{B},\omega d\mu,X_{+})$
of $L^{1}(G,\mathcal{B},\omega d\mu,X)$. Part (2) is then immediate.
As to part (1), by Theorem \ref{thm:upper-bound-function}, if $X_{+}$
is generating in $X$ there exist continuous positively homogeneous
functions $(\cdot)^{\pm}:X\to X_{+}$ and a constant $\alpha>0$ such
that $x=x^{+}-x^{-}$ and $\|x^{\pm}\|\leq\alpha\|x\|$ for all $x\in X$.
If $f\in L^{1}(G,\mathcal{B},\omega d\mu,X)$, we define $f^{\pm}(s):=(f(s))^{\pm}$
for all $s\in G$. Since the functions $(\cdot)^{\pm}:X\to X_{+}$
are continuous, the measurability of $f$ implies the measurability
of $f^{\pm}$, and the separability of $f(G)$ implies the separability
of $f^{\pm}(G)$. The inequalities $\|x^{\pm}\|\leq\alpha\|x\|\ (x\in X)$
imply $\|f^{\pm}\|_{1,\omega}\leq\alpha\|f\|_{1,\omega}<\infty$.
We conclude that $f^{\pm}\in L^{1}(G,\mathcal{B},\omega d\mu,X_{+})$.
Since $f=f^{+}-f^{-}$, the cone $L^{1}(G,\mathcal{B},\omega d\mu,X_{+})$
is generating in $L^{1}(G,\mathcal{B},\omega d\mu,X)$.
\end{proof}
Thus, in particular, if $(A,G,\alpha)$ is a pre-ordered Banach algebra
dynamical system with $A$ having a closed cone, then $L^{1}(G,A,\omega)_{+}$
is generating (proper) in $L^{1}(G,A,\omega)$ if $A_{+}$ is generating
(proper) in $A$. 

We now turn to the definition of the multiplicative structure on $L^{1}(G,A,\omega)$
if $\alpha$ is uniformly bounded. Let $(A,G,\alpha)$ be a pre-ordered
Banach algebra dynamical system, with $A$ having a closed cone, and
$\omega$ a weight on $G$. If $\alpha$ is uniformly bounded, say
$\|\alpha_{s}\|\leq C_{\alpha}$ for some $C_{\alpha}\geq0$ and all
$s\in G$, then, using the submultiplicativity of $\omega$, it is
routine to verify that 
\[
\|f*g\|_{1,\omega}\leq C_{\alpha}\|f\|_{1,\omega}\|g\|_{1,\omega}\quad(f,g\in C_{c}(G,A)).
\]
Since $C_{c}(G,A)$ is a pre-ordered algebra by Lemma \ref{lemma:twisted-convolution-positive},
it is now clear that the pre-ordered Banach space $L^{1}(G,A,\omega)$
can be supplied with the structure of a pre-ordered algebra with continuous
multiplication. If $C_{\alpha}=1$ (i.e., if $\alpha$ lets $G$ act
as bipositive isometries on $A$), then $L^{1}(G,A,\omega)$ is a
pre-ordered Banach algebra. When $C_{\alpha}\neq1$, as is well known,
there is an equivalent norm on $L^{1}(G,A,\omega)$ such that it becomes
a Banach algebra, which is a pre-ordered Banach algebra when endowed
with the same cone $L^{1}(G,A,\omega)_{+}$. In \cite[Theorem 5.17]{CPII}
it was shown that such a Banach algebra norm can be obtained from
a topological isomorphism between $L^{1}(G,A,\omega)$ and the crossed
product $\crossedprod$ for a suitable choice of $\mathcal{R}$. In
Theorem \ref{thm:Choosing-R-correctly-Crossed-Products-are-beurling}
below, we show that in the ordered context this topological isomorphism
is bipositive.
\begin{defn}
\label{def:generalized-Beurling}Let $(A,G,\alpha)$ be a pre-ordered
Banach algebra dynamical system, with $A$ having a closed cone and
$\alpha$ uniformly bounded. Let $\omega$ be a weight on $G$. The
pre-ordered Banach space $L^{1}(G,A,\omega)$ endowed with the continuous
multiplication induced by the twisted convolution on $C_{c}(G,A)$,
given by 
\[
[f*g](s):=\int_{G}f(r)\alpha_{r}(g(r^{-1}s))\, dr\quad(f,g\in C_{c}(G,A),\ s\in G),
\]
will be denoted by $\BeurlingTypeAlg$ and called a \emph{pre-ordered
generalized Beurling algebra.}
\end{defn}
{}

We note that if $A=\mathbb{R}$, the pre-ordered generalized Beurling
algebra $\BeurlingTypeAlg$ reduces to a classical Beurling algebra,
which is a true Banach algebra.

Let $(A,G,\alpha)$ be a pre-ordered Banach algebra dynamical system,
with $A$ having a closed cone. The following definition shows how
to induce a continuous covariant representation of $(A,G,\alpha)$
from a positive bounded representation of $A$. Applying this construction
to the left regular representation of $A$, and choosing (for instance)
$\mathcal{R}$ to be the singleton containing this continuous covariant
representation, yields the desired topological isomorphism (cf.\,\cite[Theorem 5.13]{CPII}).
We keep track of possible order structures in order to show later
that this topological isomorphism is bipositive.
\begin{defn}
\label{def:induced_rep_and_translation_rep}Let $(A,G,\alpha)$ be
a pre-ordered Banach algebra dynamical system, with $A$ having a
closed cone, and let $\pi:A\to B(X)$ be a positive bounded representation
of $A$ on a pre-ordered Banach space $X$ with a closed cone. We
define the induced algebra representation $\tilde{\pi}$ and left
regular group representation $\Lambda$ on the space $X^{G}$ of all
functions from $G$ to $X$ by the formulae:
\begin{eqnarray*}
[\tilde{\pi}(a)h](s) & := & \pi(\alpha_{s}^{-1}(a))h(s),\\
(\Lambda_{r}h)(s) & := & h(r^{-1}s),
\end{eqnarray*}
where $h:G\to X$, $r,s\in G$ and $a\in A$. 
\end{defn}
It is easy to see that $(\tilde{\pi},\Lambda)$ is covariant, and
positive if $X^{G}$ is endowed with the cone $X_{+}^{G}$. If $\alpha$
is uniformly bounded, then $(\tilde{\pi},\Lambda)$ yields a continuous
covariant representation of $A$ on $L^{1}(G,X,\omega)$ such that
$\|\Lambda_{r}\|\leq\omega(r)$ for all $r\in G$, and if $\pi$ is
non-degenerate, so is $(\tilde{\pi},\Lambda)$ \cite[Corollary 5.9]{CPII}.
Hence, if $A$ has a bounded approximate left or right identity, then,
with $\lambda:A\to B(A)$ denoting the left regular representation
of $A$, $(\tilde{\lambda},\Lambda)$ is a positive non-degenerate
continuous covariant representation of $(A,G,\alpha)$ on $L^{1}(G,A,\omega)$. 

Let $(A,G,\alpha)$ be a pre-ordered Banach algebra dynamical system,
with $A$ having a closed cone and a (not necessarily positive) bounded
approximate right identity. Let $\omega$ be a weight on $G$ and
$\mathcal{R}$ a uniformly bounded class of non-degenerate continuous
covariant representations of $(A,G,\alpha)$ on Banach spaces, such
that $\sup_{(\pi,U)\in\mathcal{R}}\|U_{r}\|\leq\omega(r)$ for all
$r\in G$. If $(\tilde{\lambda},\Lambda)$ is $\mathcal{R}$-continuous,
for instance if $\mathcal{R}=\{(\tilde{\lambda},\Lambda)\}$, then
the integrated form $\tilde{\lambda}\rtimes\Lambda:C_{c}(G,A)\to B(L^{1}(G,A,\omega))$
is faithful, and hence the seminorm $\sigma^{\mathcal{R}}$ is actually
a norm on $C_{c}(G,A)$ and is equivalent to $\|\cdot\|_{1,\omega}$
\cite[Theorem 5.13]{CPII}. Furthermore, $\tilde{\lambda}\rtimes\Lambda$
extends to a topological embedding $(\tilde{\lambda}\rtimes\Lambda)^{\mathcal{R}}:\crossedprod\to B(L^{1}(G,A,\omega))$
\cite[Theorem 5.13]{CPII}. Since the norms $\sigma^{\mathcal{R}}$
and $\|\cdot\|_{1,\omega}$ are equivalent, the topological isomorphism
between $\crossedprod$ and $\BeurlingTypeAlg$ which is the identity
on the mutual dense subspace $C_{c}(G,A)$ is bipositive by construction,
as the cones of both spaces are the closure of $C_{c}(G,A)$. Since
the non-degenerate $\mathcal{R}$-continuous covariant representation
$(\tilde{\lambda},\Lambda)$ is positive, so is the topological embedding
$(\tilde{\lambda}\rtimes\Lambda)^{\mathcal{R}}:\crossedprod\to B(L^{1}(G,A,\omega))$
by Lemmas \ref{lem:integrated-froms-positive} and \ref{lem:bounded-positive-operator-on-ordered-noremed-space-implies-extension-is-positive}. 

Under the assumption that, in fact, $A$ has a positive bounded approximate right
identity, we claim that the positive topological embedding $(\tilde{\lambda}\rtimes\Lambda)^{\mathcal{R}}:\crossedprod\to B(L^{1}(G,A,\omega))$
is bipositive. Identifying $\crossedprod$ with $\BeurlingTypeAlg$
through the above bipositive topological isomorphism, the topological
embedding $(\tilde{\lambda}\rtimes\Lambda)^{\mathcal{R}}$ is conjugate
to the left regular representation $\lambda:\BeurlingTypeAlg\to B(\BeurlingTypeAlg)$
through the bipositive map $\hat{\cdot}:\BeurlingTypeAlg\to\BeurlingTypeAlg$,
determined by $\hat{h}(s):=\alpha_{s}(h(s))$ for $h\in C_{c}(G,A)$
and $s\in G$ \cite[Remark 5.16]{CPII}. We denote the inverse of
$\hat{\cdot}$ by $\check{\cdot}$. With $(u_{i})\subseteq\BeurlingTypeAlg_{+}$
a positive approximate right identity (which exists by Proposition
\ref{prop:ordered-crossed-prod-has-positive-bounded-approx-id} and
the fact that $\crossedprod$ is bipositively topologically isomorphic
to $\BeurlingTypeAlg$, as described above), if $f\in\BeurlingTypeAlg$
is such that $(\tilde{\lambda}\rtimes\Lambda)^{\mathcal{R}}(f)\geq0$,
then 
\[
0\leq\left((\tilde{\lambda}\rtimes\Lambda)^{\mathcal{R}}(f)\check{u_{i}}\right)^{\wedge}=\lambda(f)u_{i}=f*u_{i}\to f.
\]
Since $\BeurlingTypeAlg_{+}$ is closed by construction, we obtain
$f\in\BeurlingTypeAlg_{+}$, and therefore the claim that $(\tilde{\lambda}\rtimes\Lambda)^{\mathcal{R}}:\crossedprod\to B(L^{1}(G,A,\omega))$
is a bipositive topological embedding follows.

We hence obtain the following ordered version of \cite[Theorem 5.17]{CPII}:
\begin{thm}
\label{thm:Choosing-R-correctly-Crossed-Products-are-beurling}Let
$(A,G,\alpha)$ be a pre-ordered Banach algebra dynamical system,
with $A$ having a closed cone and a \textup{(}not necessarily positive\textup{)}
bounded approximate right identity. Let $\alpha$ be uniformly bounded
and $\omega$ be a weight on $G$. Let the positive non-degenerate
continuous covariant representation $(\tilde{\lambda},\Lambda)$ of
$(A,G,\alpha)$ on $L^{1}(G,A,\omega)$ be as yielded by Definition
\ref{def:induced_rep_and_translation_rep}\textup{.} Then the pre-ordered
generalized Beurling algebra $\BeurlingTypeAlg$ and the pre-ordered
crossed product $\crossedprod$ with $\mathcal{R}:=\{(\tilde{\lambda},\Lambda)\}$
are bipositively topologically isomorphic via an isomorphism that
is the identity on $C_{c}(G,A)$.

Furthermore, the map $\tilde{\lambda}\rtimes\Lambda:C_{c}(G,A)\to B(L^{1}(G,A,\omega))$
extends to a positive topological embedding of $\crossedprod$ into
$B(L^{1}(G,A,\omega))$, and this extension is bipositive if $A$
has a positive bounded approximate right identity.

If $A$ has a $1$-bounded right approximate identity, $\alpha$ lets
$G$ act as isometries on $A$ and \textup{$\inf_{W\in\mathcal{Z}}\sup_{r\in W}\omega(r)=1$,
}with $\mathcal{Z}$ denoting a neighbourhood base of $e\in G$ of
which all elements are contained in a fixed compact set\textup{,}
then the bipositive topological isomorphism between $\crossedprod$
and $\BeurlingTypeAlg$ and the above positive embedding of $\crossedprod$
into $B(L^{1}(G,A,\omega))$ are both isometric. 
\end{thm}
{}

The following result gives some properties of the cones of pre-ordered
generalized Beurling algebras. Here application of Theorem \ref{thm:bochner-cones-properties}
yields stronger conclusions on the structure of the cone $\BeurlingTypeAlg_{+}$
than can be concluded from the more generally applicable Theorem \ref{thm:ordered-crossed-prod-properties}.
\begin{thm}
\label{thm:order-structures-of-beurling}Let $(A,G,\alpha)$ be a
pre-ordered Banach algebra dynamical system, with $A$ having a closed
cone and a \textup{(}not necessarily positive\textup{)} bounded
approximate right identity. Let $\alpha$ be uniformly bounded and
$\omega$ be a weight on $G$.

If the cone $A_{+}$ is generating \textup{(}proper\textup{)} in
$A$, then the cone $\BeurlingTypeAlg_{+}$ is generating \textup{(}proper\textup{)}
in $\BeurlingTypeAlg$. 

Furthermore, if $A$ is a Banach lattice algebra, then the pre-ordered
generalized Beurling algebra $\BeurlingTypeAlg$, viewed as pre-ordered
Banach space, is a Banach lattice. If, in addition, $\alpha$ lets
$G$ act as bipositive isometries on $A$, the pre-ordered generalized
Beurling algebra $\BeurlingTypeAlg$ is a Banach lattice algebra.\end{thm}
\begin{proof}
The conclusions on $\BeurlingTypeAlg_{+}$ being generating or proper
follow immediately from Theorem \ref{thm:bochner-cones-properties}.

If $A$ is a Banach lattice algebra, then $(C_{c}(G,A),C_{c}(G,A_{+}))$
with the norm $\|\cdot\|_{1,\omega}$ is a normed vector lattice.
Therefore, by \cite[Corollary 2, p.\,84]{Schaefer}, $\BeurlingTypeAlg$
is a Banach lattice. If $\alpha$ lets $G$ act as bipositive isometries
on $A$, then $\BeurlingTypeAlg$ is also a pre-ordered Banach algebra
as a consequence of Lemma \ref{lem:completion-of-normed-algebra-isbanach-algebra}
and the discussion preceding Definition \ref{def:generalized-Beurling}.
Therefore $\BeurlingTypeAlg$ is a Banach lattice algebra. 
\end{proof}
Through an application of Theorem \ref{thm:Ordered-general-correspondence},
we can now adapt \cite[Theorem 5.20]{CPII} to the ordered context,
and give an explicit description of the positive non-degenerate bounded
representations of pre-ordered generalized Beurling algebras $\BeurlingTypeAlg$
on pre-ordered Banach spaces with closed cones in terms of the positive
non-degenerate continuous covariant representations of $(A,G,\alpha)$
on such spaces, where the group representation is bounded by a multiple
of $\omega$. The result is as follows:
\begin{thm}
\label{thm:continuous-non-deg-covars-are-R-continuous}Let $(A,G,\alpha)$
be a pre-ordered Banach algebra dynamical system, with $A$ having
a closed cone, a \textup{(}not necessarily positive\textup{)} bounded
approximate right identity and a positive bounded approximate left
identity. Let $\alpha$ be uniformly bounded and $\omega$ a weight
on $G$. Then the following maps are mutual inverses between the positive
non-degenerate continuous covariant representations $(\pi,U)$ of
$(A,G,\alpha)$ on a pre-ordered Banach space $X$ with closed cone,
satisfying $\|U_{r}\|\leq C_{U}\omega(r)$ for some $C_{U}\geq0$
and all $r\in G$, and the positive non-degenerate bounded representations
$T:\BeurlingTypeAlg\to B(X)$ of the pre-ordered generalized Beurling
algebra $\BeurlingTypeAlg$ on $X$: 
\[
(\pi,U)\mapsto\left(f\mapsto\int_{G}\pi(f(r))U_{r}\, dr\right)=:T^{(\pi,U)}\quad(f\in C_{c}(G,A)),
\]
determining a positive non-degenerate bounded representation $T^{(\pi,U)}$
of the pre-ordered generalized Beurling algebra $\BeurlingTypeAlg$,
and, 
\[
T\mapsto\left(\begin{array}{l}
a\mapsto\textup{SOT-lim}_{(V,i)}T(z_{V}\otimes au_{i}),\\
s\mapsto\textup{SOT-lim}_{(V,i)}T(z_{V}(s^{-1}\cdot)\otimes u_{i})
\end{array}\right)=:(\pi^{T},U^{T}),
\]
where $\mathcal{Z}$ is a neighbourhood base of $e\in G$, of which
all elements are contained in a fixed compact subset of $G$, $z_{V}\in C_{c}(G,A)$
is chosen such that $z_{V}\geq0$ is supported in $V\in\mathcal{Z}$,
$\int_{G}z_{V}(r)dr=1$, and $(u_{i})$ is any positive bounded approximate
left identity of $A$.

Furthermore, if $A$ has an $M$-bounded \textup{(}not necessarily
positive\textup{)} approximate left identity, then the following
bounds for $T^{(\pi,U)}$ and $(\pi^{T},U^{T})$ hold:
\begin{enumerate}
\item $\|T^{(\pi,U)}\|\leq C_{U}\|\pi\|$,
\item \textup{$\|\pi^{T}\|\leq\left(\inf_{V\in\mathcal{Z}}\sup_{r\in V}\omega(r)\right)\|T\|$,}
\item $\|U_{s}^{T}\|\leq M\left(\inf_{V\in\mathcal{Z}}\sup_{r\in V}\omega(r)\right)\|T\|\omega(s)\quad(s\in G)$.
\end{enumerate}
\end{thm}
In the case where $(A,G,\alpha)=(\mathbb{R},G,\textup{triv})$ with
a weight $\omega$ on $G$, by Theorem \ref{thm:order-structures-of-beurling}
we obtain the (here rather obvious fact) fact that the classical Beurling
algebra $L^{1}(G,\omega)$ is a Banach lattice algebra. Furthermore,
Theorem \ref{thm:continuous-non-deg-covars-are-R-continuous} gives
a bijection between the positive strongly continuous group representations
of $G$ on pre-ordered Banach spaces with closed cones that are bounded
by a multiple of $\omega$, and the positive non-degenerate bounded
representations of $L^{1}(G,\omega)$ on such spaces. We hence obtain
the following adaptation of \cite[Corollary 5.22]{CPII} to the ordered
context:
\begin{cor}
\label{cor:classical-L1-result}Let $\omega$ be a weight on $G$.
Let $(z_{V})$ be as in Theorem \ref{thm:continuous-non-deg-covars-are-R-continuous}.
The maps 
\[
U\mapsto\left(f\mapsto\int_{G}f(r)U_{r}\, dr\right)=:T^{U}\quad(f\in C_{c}(G)),
\]
determining a positive non-degenerate bounded representation $T^{U}$
of the ordered Beurling algebra $L^{1}(G,\omega)$, and 
\[
T\mapsto\left(s\mapsto\textup{SOT-lim}_{V}T(z_{V}(s^{-1}\cdot))\right)=:U^{T}
\]
are mutual inverses between the positive strongly continuous group
representations $U$ of $G$ on a pre-ordered Banach space\textup{
$X$} with closed cone, satisfying $\|U_{r}\|\leq C_{U}\omega(r)$,
for some $C_{U}\geq0$ and all $r\in G$, and the positive non-degenerate
bounded representations $T:L^{1}(G,\omega)\to B(X)$ of the ordered
Beurling algebra $L^{1}(G,\omega)$ on $X$\textup{.}

If the weight satisfies $\inf_{W\in\mathcal{Z}}\sup_{r\in W}\omega(r)=1$,
where $\mathcal{Z}$ is a neighbourhood base of $e\in G$, of which
all elements are contained in a fixed compact subset of $G$, then
$\|T^{U}\|=\sup_{r\in G}\|U_{r}\|/\omega(r)$ and $\|U_{r}^{T}\|\leq\|T\|\omega(r)$
for all $r\in G$.
\end{cor}
As a particular case, the uniformly bounded positive strongly continuous
representations of $G$ on a pre-ordered Banach space $X$ with a
closed cone are in natural bijection with the positive non-degenerate
bounded representations of $L^{1}(G)$ on $X$; this also follows
from \cite[Assertion VI.1.32]{Helemski}.

Finally, we note that \cite[Theorem 8.3]{CPII} gives a bijection
between the non-degenerate bounded anti-representations of $\BeurlingTypeAlg$
on Banach spaces, for a Banach algebra dynamical system $(A,G,\alpha)$
where $A$ has a bounded two-sided approximate identity and $\alpha$
is uniformly bounded, and suitable (not covariant!) pairs $(\pi,U)$
of anti-representations of $A$ and $G$. As done above for \cite[Theorem 5.20]{CPII},
an ordered version can be derived from this, but this is left to the
reader for reasons of space.

\section*{Acknowledgements}

The authors would like to thank Anthony Wickstead for helpful comments
leading to Remark \ref{rem:wickstead}. Messerschmidt's research was
supported by a Vrije Competitie grant of the Netherlands Organisation
for Scientific Research (NWO).

\bibliographystyle{amsplain}
\bibliography{bibliography}
 
\end{document}